\documentclass{amsart}
\usepackage{graphicx} 
\usepackage{amssymb}
\usepackage{amsmath}
\usepackage{csquotes}
\usepackage{tabularx}
\usepackage{hyperref}
\usepackage{mathtools}
\usepackage{rotating}
\usepackage{tikz-cd}
\usetikzlibrary{trees}
\usepackage[a4paper, total={6in, 8in}]{geometry}
\usepackage{url}
\usepackage{theoremref}
\usepackage[T1]{fontenc}
\usepackage{mathpazo}
\usepackage{comment}
\usepackage{eulervm}
\usepackage{epigraph}
\usepackage[style=alphabetic, maxbibnames=99]{biblatex}
\DeclareMathOperator{\gr}{gr}
\newcommand{\Syl}{\mathrm{Syl}}
\newcommand{\GL}{\mathrm{GL}}
\theoremstyle{definition}
\newtheorem{Def}{Definition}[section]
\newtheorem{Prop}[Def]{Proposition}
\newtheorem{Rem}[Def]{Remark}

\newtheorem{Lem}[Def]{Lemma}
\newtheorem{Thm}[Def]{Theorem}
\newtheorem{Thm-Def}[Def]{Theorem-Definition}
\newtheorem{Kor}[Def]{Corollary}

\title{The Chow Ring of the 2-Sylow Subgroup of \(\GL(4,2)\)}
\author{Alexander Ziegler}
\bibliography{sources}
\begin{document}

\maketitle

\begin{abstract}
This paper provides a computation of the mod 2 Chow ring of the motivic étale classifying space
of the finite group \(\Syl_2(\GL(4,2))\). It outlines a general computation strategy, adapted from work by Burt Totaro, that has been largely automated by the author. This strategy can be used to compute more examples
of Chow rings of motivic étale classifying spaces quickly.
\end{abstract}

\section{Introduction}\label{sec:intro}
In \cite{toold}, Totaro assigns to every linear algebraic group \(G\) over a base field \(k\)
an ind-scheme \((U_i/G)_{i\in\mathbb{N}}=:BG\), the motivic étale classifying space of \(G\), which is ind-smooth over \(k\).
He then defines an associated
Chow ring \(CH^*(BG)\) by taking the Chow ring of each \(U_i/G\) and showing that for a fixed \(j\) the Chow groups \(CH^j(U_i/G)\) become constant for all \(i\gg j\). In this context we will care mostly about finite groups, which are interpreted as discrete linear algebraic groups over \(\mathbb{C}\). In \cite[Section 4]{MorVo}, Morel and Voevodsky independently came up with another definition of the motivic étale classifying space which is equivalent in their motivic homotopy
category.

There are many interesting computations of Chow rings of motivic étale classifying spaces covered in the literature. If we focus on finite groups, there are \begin{itemize}
    \item a computation of \(CH^*(BA)/2\) with \(A\) a finite abelian group by Totaro in \cite{toold},
    \item a computation of \(CH^*(BS_n)/2\) with \(S_n\) the \(n\)-th symmetric group by Totaro in \cite{toold},
    \item a computation recipe for \(CH^*(BS_n)\) by Guillot in \cite{chev2},
    \item a computation of \(CH^*(B\GL(n,q))/\ell\) with \(\ell>2\) a prime not dividing \(q\) by Guillot in \cite{chev1},
    \item a computation recipe for \(CH^*(B\mathrm{G}(k))/\ell\) for most finite Chevalley groups \(\mathrm{G}(k)\) with \(\mathrm{char}(k)\) and \(\ell\) coprime, \(\ell > 2\) and, depending on the type, other hypotheses by Guillot in \cite{chev2},
    \item a computation of \(CH^*(B2^{1+4}_+)/2\) in \cite{Yagita2016} and \(CH^*(B2^{1+2}_+)/2\) in \cite{spec} by Yagita,
    \item a computation of \(CH^*(Bp^{1+2}_+)/p\) and \(CH^*(Bp^{1+2}_-)/p\) for all primes \(p\) in \cite{yagita2012} by Yagita.
\end{itemize}
These computations make use of the special structural conditions of the groups, being abelian, symmetric, Chevalley or extraspecial.
It would be very desirable to have a unified approach applicable to arbitrary finite groups as in the case of group cohomology.

In his book \cite{To14}, Totaro develops such an approach to compute \(CH^*(BG)/p\) for finite groups \(G\) that admit a faithful (complex) representation of low degree or that at least admit a representation which splits into only a few direct summands of low degree. Using this, he provides in \cite[Section 13]{To14} computations of \(CH^*(BG)/2\) for all 14 groups of order \(16\), and in \cite[Section 14]{To14} of \(CH^*(BG)/p\) for all 15 groups of order \(p^4\) with \(p=3\) as well as 13  of the 15 groups of order \(p^4\) with \(p\geq 5\). 

In his thesis \cite[Chapter 3]{hphd}, Hemminger pushes
this method even further, covering many cases of groups of order 32 by employing universal polynomials as outlined in \autoref{sec:poly} and restriction maps. What connects all these cases is that the groups have faithful representations that split into direct summands such that the degree splits as \(1+1,2,1+2,2+2\) for 2-groups or \(p,1+p\) for \(p\)-groups with \(p\geq 3\).

In this paper we outline a way to push this method even further and deal with \(2\)-groups that admit an irreducible faithful representation of degree 4. This range includes the remaining groups of order 32 not covered by Hemminger and many groups of order 64.  We present this strategy by computing an explicit presentation of \(CH^*(B\Syl_2\GL(4,2))/2\). 
This might not seem like much of an improvement, but many interesting pathologies can occur in this range, in particular relations that do not come from universal polynomials in the sense of \autoref{sec:poly} and generators that might not be Chern classes. This means in particular that our method gives us a much more concrete understanding of the algebraic cycles beyond their interpretation as characteristic classes. Success of our method is dependent on algebraic cycles being detected on some proper subgroup. From our experience, this is the case for most groups that admit a faithful representation of degree 4.

We can drop the assumptions on faithful representations if we are only interested in the Chow groups of \(p\)-groups
up to degree \(p+1\). In this case we even get the integral Chow groups, making progress on a part of question (2) in \cite[Chapter 18]{To14}. Here \(p+1\) is particularly interesting, as it need not be generated by Chern classes. For the degrees up to \(p-1\) we do not even need the detection assumption from the paragraph above. This is, however, only interesting for \(p>2\).

Using our method, one can now check several open questions by Totaro on many finite groups, for example questions (3) and (9) in \cite[Chapter 18]{To14}. Generally, it would be of interest to gather many examples of pathologies, such as non-injectivity, non-surjectivity of the cycle class map, failure of the Künneth formula or Chow rings that are not generated by (transferred) Chern classes. At the end of this introduction there is also a note on progress on a conjecture of Totaro that his map \(CH^*(BG)\to MU^*(BG)\otimes\mathbb{Z}\) from \cite{toold} is an isomorphism. Some of these questions lie outside the range of groups where we can determine the full Chow ring but as explained above we
can still witness interesting behavior in degrees up to \(p+1\).

To put this quest for pathologies into context, we should mention that so far only one finite group, the extraspecial group 
\(2_+^{1+6}\), whose Chow ring is not generated by transferred Chern classes, has been identified. This is a result by Guillot in \cite{add}. Totaro shows in \cite{ck1} that there are groups of order \(p^5\) for \(p>2\) and of order 64 that do not fulfill 
the Künneth formula i.e., there exists some smooth scheme \(Y\) such that \(CH^*(BG)\otimes CH^*(Y)\to CH(BG\times Y)\) is not an isomorphism. For \(p\geq3\), Kameko constructs in \cite{Kameko2017} groups of order \(p^{p+3}\) whose mod \(p\) cycle class map is not injective. The \(p=2\) counterpart was proven by Totaro in \cite[Theorem 15.13]{To14}.
Specializing to \(2\)-groups, these examples live in the range of groups of order \(2^5,2^6,2^7\) where our method brings new insights.

Having outlined how our work fits into a broader context, we now explain how our method works and how it differs from Totaro's and Hemminger's approach. Totaro's main computational tools are given in his book \cite{To14}. First, there is a theorem showing that \(CH^i(BG)/p\) is generated by Chern classes of representations of \(G\) if \(i\leq p\), see \thref{thm:surjgrade}.
Second, \(CH^*(BG)/p\) has non-positive Castelnuovo--Mumford regularity which implies an upper bound for the degrees
of generators and relations in a presentation of \(CH^*(BG)/p\), see \thref{thm:reg}.
Third, algebraic cycles are detected on elementary abelian subgroups up to products of cycles of low degree, see \thref{thm:detect}. The second and third results reduce his computations in loc. cit. to understanding the low degree part of \(CH^*(BG)/p\) and the first result
then reduces it to understanding Chern classes.

The novel ingredient in our refinement of Totaro's strategy is to compute the associated graded to the so-called geometric filtration on the representation ring \(R(G)\) by approximating it with the \(\gamma\)-filtration using results by Karpenko and Merkurjev in \cite{KaMe}. This is very useful because Totaro's first result mentioned in the paragraph above also ascertains that for a \(p\)-group \(G\) the natural surjection \[CH^i(BG)\to\gr^i_\mathrm{geom} R(G)\] is an isomorphism for \(i\leq p\) and in another result he even gives a description of the kernel in degree \(p+1\) in terms of group cohomology, see \thref{lem:surjgrade3}. Our strategy has been largely automated by the author in \cite{GGCR} using \cite{GAP4}, \cite{SAGEMATH} and a tiny bit of \cite{M2}. We should mention here that Béatrice Chetard and Pierre Guillot developed similar unpublished \cite{SAGEMATH}-functions.

To be a bit more precise, our approach for computing \(CH^*(B\Syl_2(\GL(4,2)))/ 2\) is to determine the complex representation rings \(R(H)\) for increasingly larger subgroups \(H\leq \Syl_2(\GL(4,2))\) including their \(\lambda\)-structures.
From this, we calculate \(\gr^j_\gamma R(H)\) with \(j=1,2\) and, for the full group, also \(j=3\). It always turns out that there exists a collection of subgroups 
\(H_i\leq H\) for which we already know 
\[\gr^j_\gamma R(H_i)=\gr^j_{\mathrm{geom}} R(H_i)\]
for small \(j\) and we check that the combined restriction map
\[\gr^j_\gamma R(H)\to\prod_i \gr^j_\gamma R(H_i)\]
is injective.
From this we get that
\[\gr^j_\gamma R(H)=\gr^j_{\mathrm{geom}} R(H)\]
and deduce \(CH^j(BH)\cong \gr^j_{\mathrm{geom}} R(H)\) via \thref{thm:surjgrade}. These computational steps can be performed for any finite group \(H\) using the functions provided in \cite{GGCR} but the restriction arguments might not always be successful, for example if the restriction
\[\gr^j_\gamma R(H)\to\prod_{H_i< H} \gr^j_\gamma R(H_i)\]
is not injective. As previously mentioned, this seems to rarely be the case.

For the following arguments, we need \(H\) to admit a faithful representation of degree at most \(4\). We use Totaro's detection result \thref{thm:detect}
to ensure that the combined restriction map to a family of subgroups on the level of mod 2 Chow rings is injective, and from this we deduce that the cycle class map is injective.
We obtain the remaining relations via this restriction map. 
The non-positive regularity of the Chow ring as applied in \thref{thm:reg} ensures that we
 have not missed any generators. In the case of \(H=\Syl_2(\GL(4,2))\) an additional computation using the injectivity of the cycle class map and the group cohomology computed in \cite[Cohomology of group number 138 of order 64]{cohom} is necessary to exclude generators above degree 4. This yields a presentation of \(CH^*(BH)/2\) in terms of Chern classes of representations. More details and 
definitions of the mentioned constructions are provided in \autoref{ch:tool}.

Matthias Wendt pointed out to us that a more conceptual computation of \(CH^*(B\Syl_2(\GL(4,2)))/2\)
might be possible by using the fact that \(\Syl_2(\GL(4,2))\) can be expressed as a wreath product
\(C_2\wr C_2^2\) as mentioned in \cite{webb}.

The structure of this text is the following.
We will use \autoref{sec:basics} to introduce the notion of Chow rings of classifying spaces and some basic facts and examples. We introduce the \(\gamma\)- and the geometric filtration on a complex representation ring and relate their associated graded algebras to the
integral Chow ring in \autoref{sec:filter}. In this section we also prove a simple proposition which justifies why we can always compute
the associated graded to the \(\gamma\)-filtration on \(R(G)\) up to an arbitrary degree \(n\) for a finite group \(G\) as implemented by the author in \cite{GGCR}. Next, in \autoref{sec:control} we introduce several powerful tools from \cite{To14}
to control Chow groups of classifying spaces in high degree.
It is a bit cumbersome to deal with the \(\gamma\)-filtration by hand, so in \autoref{sec:poly} we will introduce universal polynomials that will
streamline our calculations of the associated graded pieces to the \(\gamma\)-filtration. As mentioned above, we will not be able to verify all relations by universal polynomials.
In \autoref{sec:chowh} we compute the integral Chow ring of \(H:=\Syl_2(\GL(3,2))\) and thus marginally extend a previous computation by Totaro in \cite[Chapter 13]{To14} from mod 2 to integral coefficients. 
Similarly, we compute in \autoref{sec:chowl} the mod 2 Chow ring of a subgroup \(L\leq\Syl_2(\GL(4,2))\) 
given as a centralizer of an elementary abelian subgroup of \(\Syl_2(\GL(4,2))\)
such that we can embed \(H\to L\) as a subgroup. This recovers a computation by Hemminger in \cite[Chapter 3]{hphd}. We will later also need some integral Chow groups of \(L\), which is why we state our alternative approach for this computation.
Finally, we conclude the text in \autoref{sec:chowg} by calculating the mod \(2\) Chow ring of \(\Syl_2(\GL(4,2))\) utilizing the previously computed Chow rings of \(H\) and \(L\).

\subsection{Progress on a Conjecture of Totaro}
In \cite[Chapter 18]{To14}, Totaro assembles a range of research questions about Chow rings of classifying spaces. We hope that the
computational approach outlined in this text will make some of them accessible to a systematic search for counterexamples. Particularly relevant for this paper is the first question, which is whether Totaro's map from \cite{toold} \[CH^*(BG)\to MU^{*}(BG)\otimes\mathbb{Z}\] is an isomorphism.
Totaro states that a natural test case is the family of groups \(\Syl_p (\GL(4,p))\) for \(p\geq 3\). This paper shows how to compute
\(CH^i(BG)\) for \(G=\Syl_2 (\GL(4,2))\) and \(i\leq 2+1\), but its methods can be readily applied to other prime numbers, the only problems being the size of the computations and that the method does not determine the degrees above \(p+1\). In the documentation of the code in \cite{GGCR} the author describes how to give upper bounds on the ranks of Chow groups of  \(CH^i(BG)/3\)\ for \(G=\Syl_3 (\GL(4,3))\) and \(i=1,2,3\). Thus our method makes progress on Totaro's question.

The author elaborates on this in his thesis \cite[Chapter 6]{zphd}. We would also like to use the opportunity to correct two minor mistakes in loc. cit. Using \cite{GGCR},
one can compute that the rank of \(CH^3(BG)/3\) for \(G=\Syl_3 (\GL(4,3))\) is 19, not 18 as stated in \cite[Chapter 6]{zphd}. This is just a typo, and following the explanations presented in the text one gets the correct result. In loc. cit. it is written that \((\mathbb{Z}/3\mathbb{Z})^{11}\cong \gr^2_\gamma(R(G))\cong \gr^2_\mathrm{geom}(R(G))\) 
 but what it should say is that \(\gr^2_\mathrm{geom}(R(G))\) is a quotient of \((\mathbb{Z}/3\mathbb{Z})^{11}\cong \gr^2_\gamma(R(G))\).
 Everything else said in loc. cit. remains true even with these corrections.

\subsection{Relationship to the Author's Thesis}
This text covers material from the author's thesis, namely \cite[Chapters 2,3]{zphd}. Most parts are directly lifted with some
editing to shorten the text and some corrections. In particular, this text omits the exact computations of restriction maps between Chow groups, computations involving universal polynomials and computations in \cite{SAGEMATH}. For these details we refer the reader to \cite{zphd}.
There is a mathematical error in \cite[Lemma 3.6]{zphd} that we fixed in \autoref{subsec:centr} and the fix leaves all other parts of the thesis unaffected. We also intend this paper to accompany and justify the correctness of the author's code in \cite{GGCR}. In the documentation of the code the author explains how to obtain the computation steps involving the \(\gamma\)-filtration used in this paper using loc. cit.
\subsection{AI Statement}
AI was used in this paper exclusively for surface-level proofreading. To be precise, we used it to check spelling and grammar, check for obvious mathematical mistakes such as mismatched notation or elementary arithmetic errors and check the correctness of the bibliography. The error in \cite[Lemma 3.6]{zphd} was spotted using AI and we provide some details about that in \autoref{subsec:centr}. For this we used OpenAI's ChatGPT Model GPT-5.6 Sol. The mathematical reasoning and both the mathematical and prose writing are the author's. The author’s thesis \cite{zphd} and code \cite{GGCR}, which form the basis for this work, were completed entirely without the use of any AI tools.
 \subsection*{Acknowledgments}
This work was carried out as part of a Ph.D. project supervised by Matthias Wendt and Marcus Zibrowius.
I wish to thank them for their guidance. I am deeply indebted to Matthias Wendt, who suggested this topic and provided key ideas. I thank Pierre Guillot for sharing with me the unpublished code mentioned in the introduction. Furthermore, I thank Dave Benson, Jesper Grodal, Bob Oliver and Britta Späth for helpful discussions.
\subsection*{Funding}
The author is funded by the DFG Research Training Group 2240: \emph{Algebro-Geometric Methods in
Algebra, Arithmetic and Topology}.

 \newpage
\section{Definitions and Computational Tools}\label{ch:tool}

The goal of this section is to state the definition of a Chow ring of a classifying space
of an affine group scheme as introduced by Totaro and to state a range of tools that we need to compute it for small finite groups. 

We will keep this section very light on details. We omit the proofs and constructions that are not
necessary to perform our computations. We will always give pointers towards texts that explore these objects in more detail. In the sequel, all schemes are separated and of finite type over a fixed base field \(k\). Furthermore, \(G\) will always be an affine group scheme over \(k\). It suffices for the purposes of this text to consider the case of finite discrete groups over \(\mathbb{C}\). Representations of \(G\) will always mean \(k\)-linear representations of finite degree. Lastly, \(p\) will always denote a prime number.

\subsection{ Chow Rings of Classifying Spaces}\label{sec:basics}
This subsection will very briefly introduce Chow rings of motivic étale classifying spaces as introduced by Totaro in \cite{toold}. For a more detailed account, we strongly recommend Totaro's book \cite[Chapter 2]{To14}, as it is the main source for this subsection.

\begin{Thm}\thlabel{thm:defthm}
    For any natural number \(i\), there exists a \(k\)-vector space \(V\) of finite dimension with an action \(G\to \GL(V)\) i.e. a \(k\)-linear representation of \(G\) and a closed \(G\)-invariant subscheme \(S\leq V\)  such that \(G\) acts freely on \(V-S\), \((V-S)/G\) exists as a smooth scheme
    and \(S\) is of codimension greater than \(i\). Furthermore, the graded ring 
    \(CH^*((V-S)/G)\) is independent of the choice of \(V,S\) up to degree \(i\).
\end{Thm}
\begin{proof}
    See \cite[Theorem 2.5]{To14} for a proof. A note on the smoothness of \((V-S)/G\) is directly above loc.\ cit.
\end{proof}
\begin{Def}
    The Chow ring of the motivic étale classifying space \(B_{\mathrm{ét}}G\) is defined as
    \[CH^j(B_{\mathrm{ét}}G)\coloneq CH^j((V-S)/G)\]
    for \(j<i\) by varying \(i\) and according to \thref{thm:defthm} choosing \(V,S\). We obtain a commutative graded ring \(CH^*(B_{\mathrm{ét}}G)\).
\end{Def}

In \cite[Theorem 3.1]{toold} Totaro constructs an Atiyah--Segal style isomorphism for algebraic \(K\)-theory of the form \[R(G)_k^\wedge\to K_0(B_{\mathrm{ét}}G).\] Here \(K_0(B_{\mathrm{ét}}G)\) denotes the limit of the \(K_0((V-S)/G)\) for the approximation steps
\((V-S)/G\) defining \(B_{\mathrm{ét}}G\) as in \thref{thm:defthm} and \(R(G)_k^\wedge\) is the ring of virtual representations of \(G\) completed at the kernel of the extended degree map. The \(n\)-th Chern class of algebraic vector bundles is a natural map \(c_n\colon K_0(\_)\to CH^*(\_)\) so  representations of \(G\) admit Chern classes in \(CH^*(B_{\mathrm{ét}}G)\). Throughout our computations we make heavy use of the powerful theory for these characteristic classes. 

From now on we will write \(BG\) as both the notation for the motivic étale classifying space and
if \(k=\mathbb{C}\) also the topological classifying space of \(G(\mathbb{C})\). This should not be a problem, as they only appear in the context of group cohomology or the Chow ring.
\begin{Thm}
    For discrete cyclic groups \(C_n\) generated by \(x\in C_n\) of order \(n\)
    \[CH^*(BC_n)\cong\mathbb{Z}[c_1(\sigma_n)]/(nc_1(\sigma_n))\]
    where \(\sigma_n\) is the faithful degree 1 representation \(x\mapsto \zeta_n\).  
\end{Thm}
\begin{proof}
    This is proven in \cite[Theorem 2.10]{To14}.
\end{proof}
An important tool for our computation is the fact that \(G\mapsto CH^i(BG)\) is a contravariant functor. For subgroups this specializes to a notion of restriction, which will be heavily exploited in our computations. Take for example \(H\leq G\) to be a subgroup. Often it is easier to understand \(CH^i(BH)\) than \(CH^i(BG)\).
If we want to determine all the relations governing some generating set \(S\leq CH^i(BG)\)
we can use that every relation lies in the kernel of the restriction
\[CH^i(BG)\to CH^i(BH)\]
as it is already zero in \(CH^i(BG)\). Therefore the kernel of the restriction is an upper bound on the set of relations we are trying to determine. 

What this means in practice is that given a free module \(M\coloneqq\langle S\rangle_\mathbb{Z}\) generated by a set of symbols \(S\), a surjective map \(f\colon M\to CH^i(BG)\) and a subgroup of relations \(N\leq \ker(f)\) we know that if \(M/N\to CH^i(BG)\to CH^i(BH)\) is injective, then so is 
\(M/N\to CH^i(BG)\) giving us \(M/N\cong CH^i(BG)\).  We say in this case that we show that \(N\) contains all relations by restriction to \(H\). Often this is done with multiple groups \(H_j\) by utilizing the combined restriction map \[CH^i(BG)\to \Pi_j CH^i(BH_j).\]
Usually \(S\) will consist of products of Chern classes, which means the restriction map is fully determined by restrictions of representations and thus straightforward to compute.

There is also a notion
of transfer in the other direction, but we will not need it for the later computations.

\begin{Def}\thlabel{def:pull}
    Take \(H\to G\) to be a map of \(k\)-group schemes and \(U_G\) resp. \(U_H\)
    to be the \(V-S\) chosen in the approximation of \(BG\) resp. \(BH\) as in \thref{thm:defthm}.
    The pullback along \[(U_H\times U_G)/H\to U_G/G\] gives rise to a graded ring homomorphism preserving the degree
   \[ CH^*(BG)\to CH^*(BH).\]
    If \(H\leq G\) is a subgroup scheme we call this restriction.
\end{Def}
The following standard observation is also helpful.
\begin{Prop}\thlabel{lem:chtors}
    Let \(G\) be a finite group. Then \(CH^i(BG)\) with \(i\geq 1\) is \(\# G\)-torsion.
\end{Prop}
\begin{proof}
    This follows from \cite[Lemma 2.15 (ii)]{To14}.
\end{proof}

One more tool in our arsenal will be a comparison map to group cohomology called
the cycle class map. 
\begin{Def}
    We call the map 
    \[CH^*(BG)\to H^{2*}(B(G(\mathbb{C})),\mathbb{Z})\]
    the integral cycle class map of \(BG\).
    It doubles the degree, meaning its image is concentrated in even degrees. This map is natural with respect to pullbacks as defined in \thref{def:pull}. 
\end{Def}

This induces a map \(CH^*(BG)/n\to H^{2*}(BG,\mathbb{Z})/n\to H^{2*}(BG,\mathbb{Z}/n\mathbb{Z})\),
which we will call the modulo \(n\) cycle class map. The integral cycle class map is always an
isomorphism in degree 1.
\begin{Prop}\thlabel{prop:deg1}
    \[CH^1(BG)\cong H^2(BG,\mathbb{Z})\cong \mathrm{Hom}(G,\mathbb{C}^\times).\]
\end{Prop}
\begin{proof}
    This has been proven in \cite[Lemma 2.26]{To14}.
\end{proof}

We finish this section by describing how direct products with finite abelian groups fulfill a Künneth formula.
\begin{Lem}\thlabel{lem:elab}
    Suppose \(G\) is an affine finite type group scheme over \(\mathbb{C}\) and \(A\) is a finite abelian group. Then
    \[CH^*(B(G\times A))\cong CH^*(BG)\otimes_\mathbb{Z} CH^*(BA).\]
    In particular the integral and modulo \(p\) cycle class maps are injective for 
    finite abelian groups, and \[CH^*(B(C_{n_1}\times \dots\times C_{n_k}))\cong \frac{\mathbb{Z}[c_1(\sigma_{n_1,1}),\dots, c_1(\sigma_{n_k,k})]}{(n_1c_1(\sigma_{n_1,1}),\dots, n_kc_1(\sigma_{n_k,k}))}\]
     where the \(\sigma_{n_i,i}\) are the degree 1 representations
     \[\begin{tikzcd}
         C_{n_1}\times \dots\times C_{n_k}\ar[r,"\pi_i"]&C_{n_i}\ar[r,"\sigma_{n_i}"]&\mu_{n_i}
     \end{tikzcd}\]
     with \(\pi_i\) the projection onto the \(i\)-th coordinate and \(\mu_{n_i}\leq \mathbb{C}^\times\)
    the group of \(n_i\)-th roots of unity.
\end{Lem}
\begin{proof}
    This has been proven in \cite[Lemma 2.12]{To14}.
\end{proof}

\subsection{Filtrations on the Representation Ring}\label{sec:filter}
This section will be used to introduce two filtrations on the \(k\)-linear representation ring \(R(G)\) of an affine group scheme \(G\) over a base field \(k\). Representations will always be of finite degree and \(k\)-linear.

The \(\gamma\)-filtration is very computable and it is our main lever to achieve our computations.
The geometric filtration as introduced in \cite{To14} and independently \cite{KaMe} is a much closer approximation of the Chow ring that stems from the motivic Atiyah--Hirzebruch spectral sequence. It is a priori very difficult to compute. Luckily the two filtrations coincide in the zeroth and first 
steps. In all further steps, the \(\gamma\)-filtration is a subset of the geometric filtration. Using restriction arguments one can determine the geometric filtration from the \(\gamma\)-filtration in many cases.

The crux of this work is to compute as much as one can using the \(\gamma\)-filtration, determine from this the geometric filtration and from that \(CH^*(BG)\) in low degrees,
so that the tools from the next section take care of the high degrees. To this end we used the upcoming definitions and propositions up to and including \thref{Prop:lowfilt} to write \cite{GAP4} and \cite{SAGEMATH} code in \cite{GGCR} to compute the low degrees of the associated graded algebras to the \(\gamma\)- and geometric filtration. The success of this code is dependent on the success of the restriction arguments mentioned in the paragraph above.

For a more detailed account of \(\gamma\)-filtrations on representation rings, we recommend Chetard's article \cite{Ch20}.  In the sequel, we will denote \(n\)-fold multiples of the trivial representation of \(G\) as just the natural number \(n\), because the unique ring homomorphism \(\mathbb{Z}\to R(G)\) identifies both.
\begin{Def}
    We define a \(\lambda\)-structure on \(R(G)\) via exterior powers
    \[\lambda^i(x)\coloneq\bigwedge_{j=1}^i x \]
    where \(x\) is a representation of \(G\).
    This can be uniquely extended to a map \(R(G)\to R(G)\)
    such that \(\lambda^0(x)=1\) and
    \[\lambda^n(x+y)=\sum_{i=0}^n\lambda^i(x)\lambda^{n-i}(y)\text{ for }x,y\in R(G).\]
    Thus we get so-called \(\gamma\)-operations
    \[\gamma^i(x)\coloneq\lambda^i(x+i-1).\]
\end{Def}

There is a purely axiomatical notion of (augmented) \(\lambda\)-rings that we do not explore,
see for example \cite[Generalities]{Zi10}.
The concept of \(\gamma\)-filtrations can be generalized to this more abstract setting. 
\begin{Def}
    The augmentation ideal \(\mathrm{I}\) of \(R(G)\) is defined to be the kernel of the extended degree map
    \[\mathrm{deg}\colon R(G)\to\mathbb{Z}.\]
\end{Def}
\begin{Def}
    The \(n\)-th \(\gamma\)-ideal \(\Gamma^n\trianglelefteq R(G)\) is defined as the abelian group generated by the finite products \(\prod_i \gamma^{j_i}(x_i)\) with \(x_i\in I\) and \(\sum_i j_i\geq n\).
    The descending chain of \(\gamma\)-ideals defines the \(\gamma\)-filtration on \(R(G)\).
\end{Def}

\begin{Def}
    The associated graded to the \(\gamma\)-filtration is defined as the abelian group
    \[\gr_\gamma^*R(G)=\bigoplus_{i=0}^\infty\Gamma^i/\Gamma^{i+1},\] 
    which comes equipped with a ring structure because \(\Gamma^m\Gamma^n\leq\Gamma^{m+n}\).
    We call the residue class \(\gamma^i(x-\deg(x))+\Gamma^{i+1}\) the \(i\)-th Chern class
    of a representation \(x\). This gives rise to a notion of total Chern classes
    \[c_T(x)\in \gr_{\gamma}^* R(G)[[T]]\] such that 
    \[c_T(x)\coloneq\sum_i c_i(x) T^i.\] 
\end{Def}

These Chern classes behave as we expect them to, as can be seen from the following proposition.
\begin{Prop}\thlabel{prop:chernrels}
    Chern classes in \(\gr^*_\gamma R(G)\) satisfy the following identities for representations \(x,y\)
    \begin{itemize}
        \item \(c_i(x)=0\) for \(i > \deg(x)\) ,
        \item \(c_1(xy)=c_1(x)+c_1(y)\) for degree 1 representations \(x,y\),
        \item \(c_T(x+y)=c_T(x)c_T(y)\) for representations \(x,y\), which implies
        \[c_n(x+y)=\sum_{i=0}^n c_i(x)c_{n-i}(y).\]
    \end{itemize}
\end{Prop}
\begin{proof}
    This has been proven in \cite[III.§2]{RiRo}.
\end{proof}

\begin{Prop}\thlabel{prop:grtors}
     Let \(G\) be a finite group. The abelian group \(\Gamma^n/\Gamma^{n+1}\) is \(\# G\)-torsion for \(n>0\).
\end{Prop}
\begin{proof}
    This has been proven in \cite[Proposition 2.6]{Ch20}.
\end{proof}

\begin{Rem}\thlabel{rem}
Let \(G\) be a finite discrete group over \(\mathbb{C}\). Many steps in the main computations of this text boil down to checking that two terms consisting of sums of products of elements \[\gamma^i(x-\deg x)=C_i(x)\in R(G)\] coincide. Doing
this is straightforward if we know the structure of \(R(G)\) as a ring generated by some irreducible representations of \(G\) including the \(\lambda\)-structure. 
This is enough information to express the two terms
uniquely as sums of irreducible representations. Both are equal if the summands coincide up to permutation. We will treat some examples in detail in \autoref{subsec:chgdeg2}
but later we will omit these details in the interest of readability. 

What should be mentioned is that due to \thref{prop:grtors} and the fact that the \(\Gamma^i\) are finitely generated (see \cite[Proposition 2.19]{zphd})
there are only finitely many computation steps needed to determine a presentation of \(\gr_\gamma^* R(G)\) up to some chosen degree \(n\). This ensures that the functions provided in \cite{GGCR} actually terminate.
\end{Rem}
We will now define the geometric filtration on \(R(G)\) as defined in \cite[Section 4]{KaMe}.
There is also an earlier equivalent definition in \cite[Chapter 2]{To14}.
\begin{Def}\thlabel{def:geom}
    Let \(X\) be a smooth variety by which we mean a smooth integral separated scheme of finite type over a base field \(k\). We take the class in \(K_0(X)\) given by a coherent sheaf \(\mathcal{O}\) on \(X\) to be the alternating sum of terms in a finite locally free resolution of \(\mathcal{O}\).
    Define  \(F_{\mathrm{geom}}^iK_0(X)\leq K_0(X)\) to be the subgroup generated by classes of
    coherent sheaves on \(X\) whose support has codimension of at least \(i\).
    We call this the geometric filtration on \(K^0(X)\).
    Take \(U/G\coloneq(V-S)/G\) to be an approximation of \(BG\) as described in \thref{thm:defthm}
    such that \(S\) has codimension at least \(i\). By definition and \cite[Theorem 3.1]{toold} there is a restriction map 
    \(\alpha_{U/G}\colon R(G)\to K_0(BG)\to K_0(U/G).\)
    Define \(F_{\mathrm{geom}}^i\leq R(G)\) to be the preimage of \(F_{\mathrm{geom}}^iK^0(U/G)\)
    under \(\alpha_{U/G}\). It is proven in \cite[Section 4]{KaMe} that this is independent of the choice of
    \(U/G\).
    This again gives rise to an associated graded ring \(\gr^*_{\mathrm{geom}}R(G)\), whose graded pieces are the \(F_{\mathrm{geom}}^i/F_{\mathrm{geom}}^{i+1}\). 
\end{Def}
\begin{Prop}
    The geometric filtration on \(K_0(BG)\) gives rise to an associated graded \[\gr_{\mathrm{geom}}^* K_0(BG)\]
    which is naturally isomorphic to \(\gr_{\mathrm{geom}}^* R(G)\) via the Atiyah--Segal map
    \(R(G)\to K_0(BG)\) constructed by Totaro.
\end{Prop}
\begin{proof}
    The topologies induced on \(R(G)\) by the geometric filtration and the filtration by powers of the augmentation ideal coincide, see \cite[Corollary 4.8]{KaMe}. Thus the statement follows from Totaro's isomorphism \(K_0(BG)\cong R(G)^\wedge\) in \cite[Theorem 3.1]{toold}.
\end{proof}
\begin{Prop}\thlabel{Prop:lowfilt}
    \[\Gamma^0=F_{\mathrm{geom}}^0, \Gamma^1=F_{\mathrm{geom}}^1, \Gamma^2=F_{\mathrm{geom}}^2, \Gamma^i\leq F_{\mathrm{geom}}^i\text{ for } i>2\]
    in particular there is a natural map \(\gr^*_\gamma R(G)\to \gr^*_{\mathrm{geom}}R(G)\) that is a bijection in degrees 0,1
    and a surjection in degree 2. Furthermore, if the map is a bijection in all degrees up to \(i\) then it is a surjection in degree \(i+1\).
\end{Prop}
\begin{proof}
    See \cite[Corollaries 4.8 and 4.9]{KaMe}.
\end{proof}
Having set up the machinery that makes our computations run, we still need to explain how 
the associated graded to the geometric filtration approximates \(CH^*(BG)\) in low degrees.

\begin{Thm}\thlabel{thm:surjgrade}
    There exist natural surjections \(CH^i(BG)\to \gr^i_{\mathrm{geom}}R(G)\) such that 
   \[ \begin{tikzcd}
        R(G) \ar[r,"c_i"]\ar[d,"c_i"] & CH^i(BG)\ar[d]\\
        \gr^i_{\gamma}R(G)\ar[r] &\gr^i_{\mathrm{geom}}R(G)
    \end{tikzcd}\]
    commutes. If \(G\) is a discrete \(p\)-group over \(\mathbb{C}\) and \(i\leq p\) then this is an isomorphism and \(CH^i(BG)\) is generated by Chern classes.
\end{Thm}
\begin{proof}
    This is essentially \cite[Theorem 2.25]{To14}. The map \(CH^i(BG)\to \gr^i_{\mathrm{geom}}R(G)\) identifies a codimension \(i\) subvariety \(Z\leq X\) with coherent sheaf \(\mathcal{O}_Z\). Thus, the statement on Chern classes is clear for \(i=1\) and it follows from the splitting principle for all \(i>1\).
\end{proof}

The result above is not quite enough for our computations. We need some control over the kernel of the surjective map \(CH^{p+1}(BG)\to \gr^{p+1}_\mathrm{geom}R(G)\). For this we state a result due to Totaro, building on work of Yagunov in \cite{diff}, which relates the domain of the first non-trivial differential hitting \(CH^{p+1}(BG)\) to group cohomology. 

\begin{Lem}\thlabel{lem:surjgrade3}
    Suppose \(G\) is a discrete \(p\)-group over \(\mathbb{C}\). There exists a natural map 
    \[\beta P^1\colon H_{\mathrm{sing}}^3(BG,\mathbb{Z})\to CH^{p+1}(BG)\]
    such that
    \[CH^{p+1}(BG)/\beta P^1 H_{\mathrm{sing}}^3(BG,\mathbb{Z})\cong \gr^{p+1}_{\mathrm{geom}} R(G).\]
\end{Lem}
\begin{proof}
    This is proven in \cite[Lemma 15.11]{To14}. We omit the construction of \(\beta P^1\) as it will not be needed for our computations.
\end{proof}

\subsection{Controlling the Chow Ring in High Degree}\label{sec:control}

In \autoref{sec:filter} we presented the \(\gamma\)- and the geometric filtrations on \(R(G)\) which give rise to associated graded algebras that approximate \(CH^*(BG)\)
in low degrees. If \(G\) is a finite discrete group over \(\mathbb{C}\) these low degree parts of the two associated graded algebras are computable by our code \cite{GGCR}
in many cases. To get the entire mod \(p\) Chow ring, we still have to deal with the high degrees of \(CH^*(BG)/p\). Totaro provides in \cite[Chapter 6]{To14}
and \cite[Chapter 12]{To14} two powerful tools that bound the degrees of computations necessary to
determine \(CH^*(BG)/p\). The bounds roughly scale with the degree of a faithful representation of \(G\), which is why our methods are most fruitful for groups that admit faithful representations of low degree.

We will denote by \(M^{\leq d}\) the quotient of a graded module \(M\) by its elements of degree greater than \(d\). 
\begin{Thm}\thlabel{thm:detect}
    Suppose \(G\) is a finite discrete group that admits a faithful representation of degree \(n\) and its center has \(p\)-rank \(c\). For any subgroup \(H\leq G\) denote by \(C_G(H)\) the centralizer of \(H\) in \(G\). The \(\mathbb{F}_p\)-algebra homomorphism
    \[CH^*(BG)/p\to\prod_{V}CH^*(BV)/p\otimes_{\mathbb{F}_p} CH^{\leq (n-c)}(BC_G(V))/p\]
    is injective. Here the component maps are given by the pullback along the multiplication map
    \(V\times C_G(V)\to G\) and \(V\) runs over the set of elementary abelian subgroups of \(G\).
\end{Thm}
\begin{proof}
    This is proven in \cite[Theorem 12.7]{To14}.
\end{proof}

\thref{thm:detect} means in theory we need to check for relations in a presentation of \(CH^*(BG)/p\) in degrees larger than \(n-c\) only on the collection of maximal abelian subgroups of \(G\).
Putting this into practice is a bit subtle as the image of the map in \(CH^*(BV)/p\otimes_{\mathbb{F}_p} CH^{\leq (n-c)}(BC_G(V))/p\) can contain mixed terms that are non-zero but restrict to zero in both tensor factors
\(CH^*(BV)/p\) and \(CH^{\leq (n-c)}(BC_G(V))/p\). To remedy this, we weaken \thref{thm:detect} to the following corollary that lends itself well to our restriction methods.

\begin{Kor}\thlabel{Kor:detect3}
     Suppose \(G\) is a finite discrete group that admits a faithful representation of degree \(n\) and its center is of \(p\)-rank \(c\). Denote by \(Z_p(G)\leq Z(G)\) the maximal elementary abelian \(p\)-subgroup of the center of \(G\).
    Assume furthermore that for a family of subgroups \(H_i\leq G, i\in I\) such that each \(H_i\) contains \(Z_p(G)\) as a subgroup the combined restriction map
    \[CH^r(BG)/p\to \prod_{i\in I}CH^r(BH_i)/p\]
    is injective for \(r\leq n-c\). If we denote by \(V_j,j\in J\) the family of non-central elementary abelian 
    \(p\)-subgroups of \(G\) then the combined restriction map
    \[CH^*(BG)/p\to \left(\prod_{i\in I}CH^*(BH_i)/p\right)\times\left(\prod_{j\in J}CH^*(BC_G(V_j))/p\right)\]
    is injective. It would also suffice to consider a subfamily of the \(V_j\) such that a conjugate of each non-central elementary abelian subgroup lies
    in the subfamily.
\end{Kor}
\begin{proof}
 Suppose \(x\neq 0\) lies in the kernel of the map
 \[CH^*(BG)/p\to \prod_{j\in J}CH^*(BC_G(V_j))/p\]
 then it also lies in the kernel of the combined restriction map from \thref{thm:detect}
  \[CH^*(BG)/p\to \prod_{j\in J}CH^*(B(V_j\times C_G(V_j)))/p\cong \prod_{j\in J}CH^*(BV_j)/p\otimes CH^*(BC_G(V_j))/p \]
  as the multiplication map \(V_j\times C_G(V_j)\to G\) factors through \(C_G(V_j)\).
  By \thref{thm:detect} together with the previous observation, \(x\) does not lie in the kernel of the map
  \[CH^*(BG)/p\to CH^*(BV)/p\otimes CH^{\leq n-c}(BG)/p\]
  for some central elementary abelian \(p\)-subgroup \(V\leq G\).
  We know that \[CH^*(BV)/p\otimes CH^{\leq n-c}(BG)/p\] is just a polynomial ring over \(CH^{\leq n-c}(BG)/p\)
  and any polynomial 
  restricts to zero under \[CH^*(BV)/p\otimes CH^{\leq n-c}(BG)/p\to \prod_{i\in I}CH^*(BV)/p\otimes CH^{\leq n-c}(BH_i)/p\] if and only if its coefficients lie in the kernel
  of the map
  \[CH^{\leq n-c}(BG)/p\to \prod_{i\in I}CH^{\leq n-c}(BH_i)/p.\]
  This map is injective by assumption. Therefore if we inspect the commutative square
  \begin{center}
      \begin{tikzcd}
          CH^*(BG)/p\ar[r]\ar[d]& CH^*(BV)/p\otimes CH^{\leq n-c}(BG)/p\ar[d]\\
          CH^*(BH_i)/p\ar[r]&CH^*(BV)/p\otimes CH^{\leq n-c}(BH_i)/p
      \end{tikzcd}
  \end{center}
  \(x\) cannot lie in the kernel of the composition of the right vertical and the top horizontal maps meaning it cannot
  lie in the kernel of the left vertical map. Notice here that the bottom horizontal map is induced by the multiplication
  map \(H_i\times V\to H_i\) which is only well defined if \(V\leq H_i\).  
  We have assumed that \(V\) is central and that \(Z_p(G)\leq H_i\) so that is not a problem.
  To get the statement on the possibly smaller subfamily of the \(V_j\) we use that the conjugation action of \(G\) on itself
  induces the trivial action on \(CH^*(BG)\), see \cite[Section 1.1.5]{gphd}.
\end{proof}

To formulate the second tool, we have to fix some notation. 
\begin{Def}
    Let \(R^*\) be a commutative graded algebra over a field \(k\) with \(R^0\cong k\). Take \(\overline{k}\)
    to be the algebraic closure of \(k\). Define \(\sigma(R^*)\) to be the
    minimum of all the \(\sum_i(|y_i|-1)\) over all homomorphisms
    \[\overline{k}[y_1,\dots, y_n]\to R^*\otimes_k \overline{k}\]
    such that \(R^*\otimes_k \overline{k}\) is generated by elements of bounded
    degree as a module over the polynomial ring \(\overline{k}[y_1,\dots, y_n]\).
\end{Def}
Note that it would suffice to replace every instance of \(\overline{k}\) in the above
definition with \(k\) as \(\overline{k}\) is a flat \(k\)-module. This slightly weaker version is how the upcoming \thref{thm:reg} will be used throughout this text.

The following theorem (after modding out \(p\)) gives us an example of a homomorphism \[k[y_1,\dots, y_n]\to R^*\] as in the previous definition.
\begin{Thm}\thlabel{thm:faithgens}
    Suppose \(G\) has a faithful representation \(V\) of degree \(n\). Then \(CH^*(BG)\)
    is generated as a module over the graded polynomial ring \(\mathbb{Z}[c_1(V),\dots, c_n(V)]\) by elements of degrees at most \(n^2\).
\end{Thm}
\begin{proof}
   This has been proven in \cite[Theorem 5.1]{To14}.
\end{proof}

Now we can state the second powerful tool. It gives us even more control in high degrees
compared to \thref{thm:detect} but the required lower bound on these degrees is usually higher. 
\begin{Thm}\thlabel{thm:reg}
    Let \(G\) be a finite group scheme.
    Take \(y_i\) to be the generators of the polynomial algebra appearing in the definition of \(\sigma(R^*)\) such that \(\sigma(R^*)=\sum_i(|y_i|-1)\).
    The \(\mathbb{F}_p\)-algebra \[R^*\coloneq CH^*(BG)/p\] is generated by 
    elements of degree at most
    \(\max\{|y_i|,\sigma(R^*)\}\) modulo relations of degree at most 
    \[\max\{|y_i|, \sigma(R^*)+1,2\sigma(R^*)\}.\]
\end{Thm}
\begin{proof}
    This is proven in \cite[Theorem 6.4 and Theorem 6.5]{To14}.
\end{proof}
\thref{thm:reg} can be formulated as a Castelnuovo--Mumford regularity result on \(CH^*(BG)/p\), hence the title of the section. For details, see \cite[Section 6]{To14}.
\begin{Rem}\thlabel{kor:reg}
    The union of a choice of generating sets of the \(CH^*(BG)/p\) for all \(p\) that divide \(\# G\) lifts to a generating set of the integral Chow ring \(CH^*(BG)\)
    by \thref{lem:chtors}.
    Thus if \(G\) is a \(p\)-group the bound for the degrees of a generating set of \(CH^*(BG)/p\) from \thref{thm:reg} also gives a bound for the degrees of a generating set of \(CH^*(BG)\).
\end{Rem}

\subsection{Universal Polynomials}\label{sec:poly}
This subsection recalls classical results whose use in our computations is heavily inspired by Guillot's work in \cite{unipoly}. Take \(G\) to be a finite discrete group over \(\mathbb{C}\).

We add one more tool to our arsenal to make our computations more accessible. In practice computing \(\gr^*_\gamma R(G)\) is so involved that it has to be done by a computer, thus obfuscating concrete steps. In an ideal world, all our computation steps
should be human readable. The advantage of formulating our steps using universal polynomials is thus twofold:
We make our computations more understandable to humans and we give meaning to the computed relations. 

\begin{Thm}[Splitting principle]
    Let \(G\) be a group and \(R(G)\) its representation ring. There exists an extension of \(\lambda\)-rings \(R'\) over \(R(G)\) such that \(\gr_\gamma^*R(G)\to \gr_\gamma^*R'\)
    is an injection and every element of \(R(G)\) decomposes into a \(\mathbb{Z}\)-linear combination of line elements in \(R'\) where \(x\in R'\) is called a line element if \(\lambda^n(x)\) vanishes for \(n>1\).
\end{Thm}

This splitting principle and \thref{prop:chernrels} are also available in the setting of Chow rings and singular cohomology which means the following statement is applicable to Chern classes in \(CH^*(BG)\) as well as \(H^*(BG,\mathbb{Z})\) and \( \gr^*_\gamma R(G)\).

\begin{Prop}\thlabel{prop:basicpoly}
    Suppose \(x,y\) are representations of \(G\). There exist integer polynomials \(p_{i}^{\mathrm{mul}}, p_{m,i}^{\mathrm{ext}}\) such that
    \[c_T(x\otimes y)=\sum_i p_i^{\mathrm{mul}}(c_1(x),\dots, c_i(x),c_1(y),\dots, c_i(y))T^i\]
    and
    \[c_T(\lambda^m(x))=\sum_i p_{m,i}^{\mathrm{ext}}(c_1(x),\dots, c_i(x))T^i.\]
\end{Prop}
\begin{proof}
    This simply follows from \thref{prop:chernrels} together with the splitting principle and the fundamental theorem of symmetric polynomials. For more details we refer the reader to \cite[Section I.§1]{RiRo}
\end{proof}
In the following computations we will sometimes justify relations in a presentation of \(CH^*(BG)\) as coming from universal polynomials given by some tensor product \(x\otimes y\) or exterior power \(\lambda^k(x)\) of representations \(x,y\). In these cases it is feasible to compute one of the above polynomials by hand. One can then find a second description
of \(c_i(x\otimes y)\) resp. \(c_i(\lambda^k(x))\), usually by decomposing the representations \(x\otimes y,\lambda^k(x)\)
 into irreducible representations and applying the Whitney sum formula. Both descriptions of \(c_i(x\otimes y)\) resp. \(c_i(\lambda^k(x))\) have to be equal in \(CH^*(BG)\) (also in \(H^*(BG,\mathbb{Z}),gr^*_\gamma R(G)\)).

\section{The Chow Ring of the Dihedral Group \(D_8\)}\label{sec:chowh}
Before tackling the larger group of \(\Syl_2\GL(4,2)\), we calculate \(CH^*(BH)\) for \[H\coloneq\Syl_2\GL(3,2)\] which will
be useful in our later computations for arguments via restrictions to subgroups. By finite groups we now mean finite discrete groups over \(\mathbb{C}\).
We denote the dihedral group of order \(2n\) as
\[D_{2n}\coloneq \langle x,y\mid x^{n}=y^2=1, yxy^{-1}=x^{-1}\rangle.\]
 Now \(H\) can be understood as the group of
upper-triangular \(3\times 3\)-matrices over \(\mathbb{F}_2\) with \(1\) as diagonal entries and there is an isomorphism \(H\to D_8\) given by
\[\begin{pmatrix}1&1&0\\0&1&1\\0&0&1 \end{pmatrix}\mapsto x,\begin{pmatrix}1&1&0\\0&1&0\\0&0&1\end{pmatrix}\mapsto y.\]
In \cite[Section 3.1]{zphd} the author computes the integral Chow ring for any dihedral group \(D_{2n}\) of order \(2n\), generalizing the computation by Totaro in \cite[Lemma 13.2]{To14} who considered the mod-2 Chow ring for the case where \(n\) is a power of 2.

We first state a description of the representation ring of \(D_{8}\). This can be utilized to express \(CH^*(BD_{8})\) in terms of Chern classes. We use \thref{kor:reg} to show that \(CH^*(BD_{8})\)
is generated by terms of degree 2 and therefore by \thref{thm:surjgrade} generated by Chern classes. We then determine some relations governing these generators via universal polynomials as described in \autoref{sec:poly}. We show that we have found all relations by restriction to subgroups as explained in \autoref{sec:basics}.

 \begin{Prop}\thlabel{Prop:repdih}
\[R(D_{8})\cong\frac{\mathbb{Z}[a,b,V]}{( a^2-1,b^2-1,V^2-1-a-b-ab)}\]
    with \[a\colon x\mapsto -1, y\mapsto 1,\]
    \[b\colon x\mapsto 1, y\mapsto -1,\]
    \[V\colon x\mapsto \begin{pmatrix}i&0\\0& -i\end{pmatrix},y\mapsto\begin{pmatrix}0&1\\1&0\end{pmatrix}.\]
\end{Prop}
\begin{proof}
    The structure of the representation rings of the dihedral groups follows directly from their character tables which are given in \cite[Section 5.3]{serre}. For the \(\lambda\)-structure we use that the top exterior power is the determinant.
\end{proof}
In the sequel we denote for a ring \(R\) the free \(R\)-module generated by \(a_1,\dots, a_n\) as \(\langle a_1,\dots, a_n\rangle_R\) .
\begin{Prop}\thlabel{prop:chdih}
    \[CH^*(BD_{8})\cong\mathbb{Z}[c_1(a),c_1(V),c_2(V)]/(2c_1(V),2c_1(a),4c_2(V),c_1(a)c_1(V)+c_1(a)^2).\]
    Furthermore, \(c_1(V)=c_1(b)\) and the cycle class map of \(D_8\) is injective.
\end{Prop}
\begin{proof}
    By \thref{kor:reg} together with \thref{thm:faithgens}, \(CH^*(BD_{8})\) is generated by elements of degree at most 2 and by \thref{thm:surjgrade}
    it is generated by Chern classes.  Next we determine the relations in the presentation above via universal polynomials as described in \autoref{sec:poly}.  We see that \(\lambda^2(V)=b\) meaning \(c_1(V)=c_1(b)\).

    We have \(\mathrm{Hom}(D_{8},\mathbb{C}^\times)\cong \mathbb{Z}/2\mathbb{Z}\oplus \mathbb{Z}/2\mathbb{Z}\) is generated by \(a,b\). Thus by \thref{prop:deg1} \[CH^1(BD_{8})=\langle c_1(a),c_1(b)\rangle_{\mathbb{Z}/2\mathbb{Z}}=\langle c_1(a),c_1(V)\rangle_{\mathbb{Z}/2\mathbb{Z}}.\]

    Next we prove the relations in degree 2.
    The universal polynomial for the second Chern class of a tensor product tells us
    \[c_2(V\otimes V)=5c_1(V)^2+4c_2(V)\]
    and the Whitney sum formula tells us
    \[c_2(1+a+b+ab)=c_1(a)c_1(b)+c_1(a)(c_1(a)+c_1(b))+ c_1(b)(c_1(a)+c_1(b))\]
    Using the identity \(V\otimes V=1+a+b+ab\) as well as the fact that \(c_1(V)=c_1(b)\) is 2-torsion, we can conclude
    \[ 4c_2(V)=c_1(a)c_1(V)+ c_1(a)^2.\]
     It remains to show \(c_1(a)^2+c_1(a)c_1(V)=0\). We get this via the fact \(V=a\otimes V\) and the universal polynomial of a tensor product
    \[c_2(V)=c_2(a\otimes V)= c_1(a)^2+c_1(a)c_1(V)+c_2(V).\]

    Take \[C_4\cong H_1\leq D_8\] to be the subgroup generated by \(x\) and \(C_2\times C_2 \cong H_2\leq D_8\) the subgroup generated by \(x^2,y\). One easily checks that 
    \[\mathbb{Z}[c_1(a),c_1(V),c_2(V)]/(2c_1(V),2c_1(a),4c_2(V),c_1(a)c_1(V)+c_1(a)^2) \to CH^*(BH_1)\times CH^*(BH_2)\]
    is injective. By the injectivity of the restriction above
    we have found all relations and thus we get
    \[CH^*(BD_{8})\cong \mathbb{Z}[c_1(a),c_1(V),c_2(V)]/(2c_1(V),2c_1(a),4c_2(V),c_1(a)c_1(V)+c_1(a)^2).\]
    The cycle class map is injective by the injectivity of the restriction to \(H_1,H_2\) together with \thref{lem:elab}.
\end{proof}

\section{Towards detecting Cycles on a Family of Subgroups of \(\Syl_2\GL(4,2)\)}\label{sec:chowl}

In this section we take another intermediate step towards the Chow ring of \(G\coloneq\Syl_2\GL(4,2)\). We consider a family of subgroups \(L_i,K\leq G\) such that every centralizer \(C_G(V)\) of a non-central elementary abelian subgroup \(V\leq G\) is a subgroup of a conjugate of one of the \(L_i\) or \(K\). As before, finite groups in this section are considered to be finite discrete groups over \(\mathbb{C}\). By \thref{Kor:detect3} together with the facts that \(G\) has a faithful representation of degree 4 and a center of \(2\)-rank 1, we know that if the map
\[CH^j(BG)/2\to CH^j(BK)/2\times \prod_i CH^j(BL_i)/2\]
is injective in degrees \(j=1,2,3\) it is always injective. In \thref{fix},
we will see that the restriction data on \(K\) is redundant, so we can ignore it. We will say that the family \(L_i\) detects (modulo 2) cycles.

This section will consist of us making a guess for a family of subgroups \(L_i\leq G\) such that we will see in \autoref{sec:chowg}
that the \(L_i\) actually detect cycles. We will then determine the \(CH^*(BL_i)/2\) so that we can use the restriction maps to determine relations in degrees beyond 3. In fact, it will turn out that the \(L_i\) belong to two isomorphism classes, one of which is \(D_8\times C_2\) and the other we denote by \(L\). The Chow ring of \(D_8\times C_2\) can be deduced via \thref{lem:elab}
from the previous section. Therefore we only need to compute \(CH^*(BL)/2\), which is the goal of this section.

We should note here that \(CH^*(BL)/2\) was already computed by Hemminger in \cite[Chapter 3]{hphd}. Sadly, the citation is missing from the author's thesis \cite[Section 3.2]{zphd} and the author wishes to apologize for this oversight. Hemminger also uses universal polynomials and restrictions, but he does not use \(gr_{\gamma}^*R(L)\) and does not make any statements about the integral Chow groups in low degree. Later, we will need both the integral Chow groups in low degree and the modulo 2 Chow ring.

\subsection{Subgroups containing Centralizers of Elementary Abelian Groups}\label{subsec:centr}
In this subsection we identify \(G\) with the \(\mathbb{F}_2\)-matrix group
\[G\coloneq\left\{\begin{pmatrix}
    1&*&*&*\\0&1&*&*\\0&0&1&*\\0&0&0&1\
\end{pmatrix}\in\GL(4,2)\right\}.\]  We denote by \(E_{i,j}\in\GL(4,2)\) the elementary \(4\times4\)-matrix whose entries are 1 on the main diagonal and in the \((i,j)\)-th entry while all other entries are 0.

Define a family of subgroups of \(G\) 
    \[L_{(x,z)}\coloneq\left\{\begin{pmatrix}
    1&a&*&*\\0&1&*&*\\0&0&1&c\\0&0&0&1\
\end{pmatrix}\in\GL(4,2) \middle| xa+zc=0\right\}\leq G \text{ with }(x,z)\in\mathbb{F}_2^2-\{(0,0)\}\]
which are isomorphic via the map
 \(L_{(1,0)}\to L_{(0,1)}\)
 \[E_{1,3}\mapsto
E_{2,3},E_{2,3}\mapsto E_{2,4}, E_{3,4}\mapsto E_{1,2},\]
 and the map \(L_{(0,1)}\to L_{(1,1)}\)
 \[E_{1,2}\mapsto
\begin{pmatrix}1&1&1&1\\ 0&1&0&1\\0&0&1&1\\0&0&0&1\end{pmatrix},E_{2,3}\mapsto
\begin{pmatrix}1&0&1&0\\ 0&1&1&1\\0&0&1&0\\0&0&0&1\end{pmatrix}, E_{2,4}\mapsto
\begin{pmatrix}1&0&1&1\\ 0&1&1&0\\0&0&1&0\\0&0&0&1\end{pmatrix}.\]
Furthermore, define
\[L_0\coloneq\left\{\begin{pmatrix}
    1&*&*&*\\0&1&0&0\\0&0&1&*\\0&0&0&1\
\end{pmatrix}\in\GL(4,2) \right\}\leq G, K\coloneq\left\{\begin{pmatrix}
    1&*&a&*\\0&1&0&a\\0&0&1&*\\0&0&0&1\
\end{pmatrix}\in\GL(4,2) \right\}\leq G\] \[L_\infty\coloneq\left\{\begin{pmatrix}
    1&*&0&*\\0&1&0&*\\0&0&1&*\\0&0&0&1\
\end{pmatrix}\in\GL(4,2) \right\}\leq G.\]
Denote the generator of \(C_2\) as \(z\) and recall from \autoref{sec:chowh} the presentation
\[D_8\coloneq\langle x,y\mid x^4=y^2=1, yxy=x^{-1}\rangle.\] 
   We have isomorphisms  \(L_0\to D_8\times C_2\)
\[E_{1,2}\mapsto (z,\mathrm{id}), E_{1,3}\mapsto (\mathrm{id},y),E_{3,4}E_{1,3}\mapsto (\mathrm{id},x)\]
and \(K\to D_8\times C_2\)
\[E_{1,2}E_{3,4}\mapsto (z,\mathrm{id}), E_{1,2}\mapsto (\mathrm{id},y),E_{2,4}E_{1,3}E_{1,2}\mapsto (\mathrm{id},x)\]
  and  \(L_\infty\to D_8\times C_2\)
\[E_{1,2}\mapsto (\mathrm{id},y), E_{3,4}\mapsto (z,\mathrm{id}),E_{2,4}E_{1,2}\mapsto (\mathrm{id},x).\]
Checking that all of the maps given above are well-defined isomorphisms is straightforward in \cite{GAP4}.

Due to a computational error in \cite[Lemma 3.6]{zphd} found by
OpenAI's ChatGPT, 
the author missed some of the centralizers of non-central elementary abelian subgroups. The error is that on page 47 of loc. cit. the author incorrectly states
\[\begin{pmatrix}
    1&u&x&z\\0&1&v&y\\0&0&1&w\\0&0&0&1\
\end{pmatrix}\begin{pmatrix}
    1&a&d&f\\0&1&b&e\\0&0&1&c\\0&0&0&1\
\end{pmatrix}=\begin{pmatrix}
    1&a&d&f\\0&1&b&e\\0&0&1&c\\0&0&0&1\
\end{pmatrix}\begin{pmatrix}
    1&u&x&z\\0&1&v&y\\0&0&1&w\\0&0&0&1\
\end{pmatrix}\] implies \(ay=eu,cx=dw\) when it should be \(ay+dw=cx+eu\). This affects the case distinction directly afterwards.

Up to conjugation only one centralizer is missing, namely \(K\) as defined above.
Considering the centralizers up to conjugation is not a problem as the conjugation action of
\(G\) on \(CH^*(BG)\) is trivial, see \cite[Section 1.1.5]{gphd}.
We show in \thref{fix} that we can safely ignore \(K\) when applying \thref{Kor:detect3}. This means the other arguments 
in \cite[Chapter 3]{zphd} are unaffected. Furthermore, ChatGPT pointed out that \cite[Lemma 3.6]{zphd} was not optimal as we could exclude \(L_{(1,0)},L_{(0,1)}\) as possible centralizers. The key observation is that \(L_{(1,0)},L_{(0,1)}\) are centralizers of cyclic groups of order 4 not 2. 
We have adapted the Lemma in \thref{Lem:centr} accordingly but this too does not change our later computations in any way as we need all groups to detect cycles of \(G\).
\begin{Lem}\thlabel{Lem:centr}
    Let \(V\leq G\) be a non-central elementary abelian subgroup. If we consider \(V\) up to conjugation, the centralizer \(C_G(V)\) is a subgroup of at least one of the groups \(L_{(1,1)},L_0,L_{\infty}, K\).
\end{Lem}
\begin{proof}
    We begin by considering \(V\) of the form
    \[C_2^{(x,z)}\coloneqq \left\{\begin{pmatrix}
    1&0&u&*\\0&1&0&w\\0&0&1&0\\0&0&0&1\
\end{pmatrix}\in\GL(4,2) \middle| xu+zw=0\right\}\]
with \((x,z)\in\mathbb{F}_2^2-\{(0,0)\}\) in which case \(C_G(V)=L_{(x,z)}\). We can disregard the case of \((x,z)=(0,0)\) as then \(V\) is the center of \(G\).

    It suffices to show the lemma for the case of \(V\) being a non-central cyclic group generated by an element of order 2
\[\alpha\coloneqq \begin{pmatrix}
    1&u&x&z\\0&1&v&y\\0&0&1&w\\0&0&0&1\
\end{pmatrix}\]
    as for any intermediate group \(V\leq W\leq G\) we have \(C_G(W)\leq C_G(V)\). If \(u=v=w=0\),
    the previous step already tells us that  \(L_{(x,z)}\leq C_G(V)\neq G\) for some \((x,z)\neq (0,0)\) because any such \(V\) is a subgroup of an intermediate group of the form \(C_2^{(x,z)}\).
    Because \(L_{(x,z)}\) is a subgroup
    of \(G\) of index 2 we get \(L_{(x,z)}= C_G(V)\). Thus we can restrict to the case where one of \(u,v,w\) is non-zero.
    If \(v=1\) then \(u=w=0\) as otherwise \(\alpha\) has order 4. In this case the centralizer of \(\alpha\) is a subgroup of \(L_{(1,1)}\).

It remains to treat the \(v=0\) case. To compute the centralizer of \(\alpha\) we will assume \(z=0\) as multiplication by elements of the center does not change the centralizer. We will assume \(uy=xw\) as otherwise \(\alpha\) has order 4.
Note that 
\[E_{2,3}\alpha E_{2,3}^{-1}=\begin{pmatrix}
    1&u&x+u&0\\0&1&0&y+w\\0&0&1&w\\0&0&0&1\
\end{pmatrix}.\]
Considering \(\alpha\) up to conjugation with \(E_{2,3}\) our assumptions leave us with the cases

\begin{itemize}
\item \(u=w=1,v=x=y=z=0\) and \(C_G(\alpha)=K\),
\item \(u=1, v=w=x=y=z=0\) and \(C_G(\alpha)=L_0\),
\item \(w=1, u=v=x=y=z=0\) and \(C_G(\alpha)=L_\infty\).
\end{itemize}
This covers all possible cases.
\end{proof}
We would need the identification of the subgroups to explicitly construct the restriction map, but this laborious step is omitted in the interest of readability. A detailed account is given in the author's thesis \cite{zphd} and the construction is straightforward using restrictions of characters.
\begin{Lem}\thlabel{fix}
    Any cycle in the kernel of
    \[CH^i(BG)/2\to  CH^i(BL_0)/2\times CH^i(BL_{(1,1)})/2\]
    is also an element of the kernel of \(CH^i(BG)/2\to CH^i(BK)/2\).
    In particular if a family of subgroups containing \(L_0, L_{(1,1)}, K\)
    detects (modulo 2) cycles then it already detects them without \(K\).
\end{Lem} 
\begin{proof}
    We identify \(K\cong D_8\times C_2\) as above. This gives us a presentation
    \[K\cong\langle x, y, z \mid x^4=y^2=z^2=[x,z]=[y,z]=1, yxy^{-1}=x^{-1}\rangle\]
    The group \(L_0\cap K\cong C_2\times C_2\times C_2\) is the subgroup generated by 
    \(x^2, y, z\) and \(L_{(1,1)}\cap K\cong C_2\times C_2\times C_2\) is the subgroup 
    generated by \(x^2, yx, z\). Using the computation in \thref{prop:chdih} and \thref{lem:elab} it is then easy to check that
    the combined restriction map 
    \[CH^*(BK)/2\to  CH^*(B(L_0\cap K))/2\times CH^*(B(L_{(1,1)}\cap K))/2\]
    is injective. This proves the lemma.
\end{proof}
\subsection{Representations of \(L\)}\label{subsec:repl}
Denote now  by \(L\coloneq L_{(1,0)}\) the group from \thref{Lem:centr}. To compute its Chow ring we first determine \(R(L)\) including its \(\lambda\)-structure. The representation ring \(R(L)\) can be obtained from the character tables provided by \cite{GAP4} where \(L\) is given by the 
entry \((32,27)\) in the SmallGroups library \cite{SmallGrp}. For completeness' sake, we also determine concrete matrix representations that afford these characters.
 Denote for \((a,b,c)\in\mathbb{F}_2^3\) by \(f_{(a,b,c)}\colon L\to\mathbb{C}^\times\) the map
    \[\begin{pmatrix}1&0&u&*\\0&1&v&*\\0&0&1&w\\0&0&0&1\end{pmatrix}\mapsto (-1)^{au+bv+cw}.\]
    Recall from \thref{Prop:repdih} the standard representation \(V\colon D_8\to \GL(2,\mathbb{C})\).
    For \(nk^{-1}\in\mathbb{F}_2\cup\{\infty\}\) with \(n,k\in\mathbb{F}_2\) not both zero, there is a surjective homomorphism \(\varphi_{nk^{-1}}\colon L\to D_8\) given by the assignment
    \[\begin{pmatrix}1&0&u&0\\0&1&v&0\\0&0&1&0\\0&0&0&1\end{pmatrix}\mapsto y^{nv+ku},\begin{pmatrix}1&0&0&0\\0&1&0&0\\0&0&1&1\\0&0&0&1\end{pmatrix}\mapsto xy.\]
    This gives rise to a representation \(\phi_{nk^{-1}}=V\circ\varphi_{nk^{-1}}\).

\begin{Prop}\thlabel{prop:repl2}
   
    The representation ring \(R(L)\) can be expressed via the presentation
    \[R(L)\cong\mathbb{Z}[f_{(a,b,c)},\phi_{nk^{-1}}\mid(a,b,c)\in\mathbb{F}_2^3,nk^{-1}\in\mathbb{F}_2\cup\{\infty\}] /I\]
    where \(n,k\in\mathbb{F}_2\) are not both 0 and \(I\) is generated by the relations

    \begin{align*}f_{(a,b,c)}\otimes f_{(a',b',c')}&=f_{(a+a',b+b',c+c')},\\f_{(0,0,0)}&=1,\\
    f_{(ka,na,c)}\otimes \phi_{nk^{-1}}&= \phi_{nk^{-1}} \\
    \phi_{nk^{-1}}^2&=\sum_{x,y\in\mathbb{F}_2}f_{(kx,nx,y)},
    \end{align*}
    and for all \(\ell_1,\ell_2,\ell_3\in\{0,1,\infty\}\) pairwise distinct and \(a,b,c\in\mathbb{F}_2\)
    such that \(\phi_{\ell_3}\neq f_{(a,b,c)}\otimes\phi_{\ell_3}\)
    \[\phi_{\ell_1}\otimes\phi_{\ell_2}=\phi_{\ell_3}+f_{(a,b,c)}\otimes\phi_{\ell_3}.\]
     Furthermore, the \(\lambda\)-structure on \(R(L)\) is determined by
     \[\lambda^2(\phi_0)=f_{(1,0,1)}, \lambda^2(\phi_1)=f_{(1,1,1)},\lambda^2(\phi_\infty)= f_{(0,1,1)}.\] 
\end{Prop}
\begin{proof}
    This is easy to verify using the character table of \(L\) provided by \cite{GAP4}.
\end{proof}
\subsection{Determining Degree 1}
We will now use the tools presented in \autoref{sec:filter} and \autoref{sec:poly} to compute \(CH^i(BL)\) for \(i\leq 2\) together with the kernel of a combined restriction map to the subgroups
\(H_i\leq L\) defined below. When we say that a family of subgroups \(H_i\leq L\) detects (modulo 2) cycles of \(L\) in degree \(i\), we mean that
\[CH^i(BL)/2\to \prod_{i} CH^i(BH_i)/2\]
is injective. In the sequel we denote by \(f_{(a,b,c)}\),
\(A\coloneq \phi_0,B\coloneq \phi_1,C\coloneq \phi_\infty\) the representations described in \autoref{subsec:repl}.
For \(nk^{-1}\in\mathbb{F}_2\cup\{\infty\}\) with \(n,k\) not both zero, we denote a family of subgroups
\[H_{nk^{-1}}\coloneq \left\{\begin{pmatrix} 1&0&ka&*\\0&1&na&*\\0&0&1&*\\0&0&0&1\end{pmatrix}\middle| a\in\mathbb{F}_2\right\}.\]

\begin{Lem}\thlabel{Lem:deg1l}
    \[CH^1(BL)\cong\langle c_1(A), c_1(B), c_1(C) \rangle_{\mathbb{Z}/2\mathbb{Z}}\]
    where we have \[c_1(f_{(1,0,1)})=c_1(A), c_1(f_{(1,1,1)})=c_1(B), c_1(f_{(0,1,1)})=c_1(C).\]
    Furthermore, \(H_0,H_1,H_\infty\) detect (modulo 2) cycles of degree 1.
\end{Lem}
\begin{proof}
The computation of \(CH^1(BL)\) follows from \thref{prop:deg1} together with the formulas for top Chern classes from \thref{prop:repl2}. 
The \enquote{furthermore} part can simply be checked by computing the restriction on the level of representations.
\end{proof}
\subsection{Determining Degree 2}
To compute \(CH^2(BL)\) we make use of the theory of universal polynomials as outlined in \autoref{sec:poly}. Each step is straightforward so we only
state which universal polynomials are used and do not give the concrete computations. We invite the curious reader to check our results on \(\gr^*_\gamma R(L),gr^*_\mathrm{geom} R(L)\) using the code provided in \cite{GGCR}.
\begin{Lem}\thlabel{Lem:deg2l}
    \begin{align*}
        \gr^2_{\mathrm{geom}}R(L)\cong CH^2(BL)\cong&\frac{\langle c_2(A),c_2(B),c_2(C)\rangle_{\mathbb{Z}/4\mathbb{Z}}}{\langle 2(c_2(A)+c_2(B)+c_2(C))\rangle_{\mathbb{Z}/2\mathbb{Z}}}\\&\quad\quad\quad\quad\quad \oplus\langle c_1(A)^2,c_1(B)^2,c_1(C)^2,c_1(A)c_1(C)\rangle_{\mathbb{Z}/2\mathbb{Z}}
    \end{align*}
    where we have \begin{gather*} c_1(A)c_1(B)=c_1(B)^2+c_1(C)^2+c_1(A)c_1(C), \\c_1(B)c_1(C)=c_1(A)^2+c_1(B)^2+c_1(A)c_1(C).\end{gather*}
    Furthermore, the \(H_0,H_1,H_\infty\) detect (modulo 2) cycles of degree 2.
\end{Lem}
\begin{proof}
By \thref{thm:surjgrade} \(\gr^2_{\mathrm{geom}}R(L)\cong CH^2(BL)\)
and by \thref{Prop:lowfilt} \[\gr^2_{\gamma}R(L)\to \gr^2_{\mathrm{geom}}R(L)\] is surjective.
We know from \thref{prop:repl2} that the \(f_{(a,b,c)},A,B,C\) generate \(R(L)\) so \(\gr^2_{\gamma}R(L)\) and thus also \(CH^2(BL)\) are generated by twofold products of first Chern classes of \(f_{(a,b,c)},A,B,C\) and second Chern classes of \(A,B,C\). By \thref{Lem:deg1l} it suffices to consider terms consisting of Chern classes of \(A,B,C\). Thus \(CH^2(BL)\) is generated by 
\[c_1(A)^2,c_1(A)c_1(B),c_1(A)c_1(C),c_1(B)^2,c_1(B)c_1(C),c_1(C)^2,c_2(A),c_2(B),c_2(C).\]
It remains to determine all degree 2 relations governing these generators.
We have seen in our previous computation in \thref{Lem:deg1l} that the first Chern classes are 2-torsion.
The relations for \(c_1(A)c_1(B),c_1(B)c_1(C)\) come from the  universal polynomials for the second Chern classes of tensor products 
\begin{gather*}f_{(0,1,1)}\otimes A=f_{(1,1,1)}\otimes A,\\ f_{(1,0,1)}\otimes C=f_{(1,1,1)}\otimes C.\end{gather*}

The universal polynomials for second Chern classes of tensor products
\begin{align*}
    A\otimes A&=f_{(0,0,0)}+f_{(1,0,0)}+f_{(0,0,1)}+f_{(1,0,1)},\\
    B\otimes B&=f_{(0,0,0)}+f_{(1,1,0)}+f_{(0,0,1)}+f_{(1,1,1)},\\
    C\otimes C&=f_{(0,0,0)}+f_{(0,1,0)}+f_{(0,0,1)}+f_{(0,1,1)}
\end{align*}
tell us that \(c_2(A),c_2(B),c_2(C)\) are 4-torsion.

Lastly, the universal polynomial for the second Chern class of
\[A\otimes B=C+ f_{(1,0,0)}\otimes C\]
 gives us
\[2c_2(A)+2c_2(B)=2c_2(C).\]

It can be checked that we found all relations by restriction to the \(CH^2(BH_i)\). Verifying the \enquote{furthermore} part is straightforward.
\end{proof}
\subsection{Determining \(CH^*(BL)/2\)}
We fix one more subgroup to improve our restriction arguments
\[C_2^{4,L}\coloneq \left\{\begin{pmatrix} 1&0&*&*\\0&1&*&*\\0&0&1&0\\0&0&0&1\end{pmatrix}\right\}\leq L.\]
\begin{Prop}\thlabel{Prop:chldetect}
    The combined restriction map
    \[CH^*(BL)/2\to CH^*(BC_2^{4,L})/2\times\prod_{i\in\{0,1,\infty\}} CH^*(BH_i)/2\]
    is injective. 
\end{Prop}
\begin{proof}
    We will first show that the subgroups \(H_0,H_1,H_\infty, C_2^{4,L}\) detect (mod 2) cycles. We know that they detect cycles in degrees 1,2 by \thref{Lem:deg1l} and \thref{Lem:deg2l}. By \thref{Kor:detect3} \(H_0,H_1,H_\infty, C_2^{4,L}\) together
with all centralizers of non-central elementary abelian subgroups of \(L\) detect cycles. Thus we only need to
show that every such centralizer is a subgroup of one of the \(H_0,H_1,H_\infty, C_2^{4,L}\). It suffices to show this for all cyclic subgroups \(\langle l\rangle\) of prime order as for any intermediate group \(\langle l\rangle \leq V\leq L\) we have \(C_L(V)\leq C_L(\langle l\rangle).\)

Take a generic element
\[l\coloneq \begin{pmatrix} 1&0&u&x\\0&1&v&y\\ 0&0&1&w\\0&0&0&1\end{pmatrix}\in L.\]
The group generated by \(l\) is of prime order unless \(u=w=1\) or \(v=w=1\).
It is non-central unless \(u=v=w=0\). This leaves three remaining cases
\begin{itemize}
    \item \(u=v=0,w=1\) in which case \(C_L(\langle l\rangle)=H_0\cap H_\infty\),
    \item \(u=1,w=0\) in which case \(C_L(\langle l\rangle)\leq C_2^{4,L}\),
    \item \(v=1,w=0\) in which case \(C_L(\langle l\rangle)\leq C_2^{4,L}\).
\end{itemize}
Note that \(H_0,H_1,H_\infty, C_2^{4,L}\) all contain \(Z(L)\) as a subgroup. We can conclude from \thref{Kor:detect3} that 
\[CH^*(BL)/2\to CH^*(BC_2^{4,L})/2\times\prod_{i\in\{0,1,\infty\}} CH^*(BH_i)/2\]
is injective.
\end{proof}
Because the (mod 2) cycle class map is injective for \(D_8\) and \(C_2^4\) we get:
\begin{Kor}\thlabel{cor:lcyc}
    The (mod 2) cycle class map \(CH^*(BL)/2\to H^*(BL,\mathbb{Z}/2\mathbb{Z})\) is injective.
\end{Kor}
Now we bring everything together.
\begin{Prop}\thlabel{prop:chowl}
    The Chow ring \(CH^*(BL)/2\) can be expressed via a presentation
     \[CH^*(BL)/2\cong \mathbb{F}_2[c_i(A),c_i(B),c_i(C)\mid i=1,2]/I\] 
where \(I\) is the ideal generated by relations
in degree 2:
\begin{gather*}
    c_1(A)c_1(B)+c_1(B)^2+c_1(A)c_1(C)+c_1(C)^2,\; c_1(A)^2 + c_1(B)^2 + c_1(A)c_1(C)+c_1(B)c_1(C),
\end{gather*}
in degree 3:
\begin{gather*}
    c_1(A)c_2(A)+c_1(B)c_2(A)+c_1(C)c_2(A)+c_1(A)c_2(B)\\ 
    +c_1(B)c_2(B)+c_1(C)c_2(B)+c_1(A)c_2(C) + c_1(B)c_2(C)+c_1(C)c_2(C),
\end{gather*}
in degree 4:
\begin{gather*}
    c_1(B)c_1(C)c_2(A)+c_1(A)c_1(C)c_2(B)+c_1(B)^2c_2(C)+ c_1(A)c_1(C)c_2(C) \\
    + c_1(C)^2c_2(C) +c_2(A)^2 + c_2(B)^2 + c_2(C)^2.
\end{gather*}
\end{Prop}
\begin{proof}

The degree 4 representation \(A+B\) is faithful. By \thref{thm:faithgens}
we get that \(CH^*(BL)/2\) is generated by elements of bounded degree as a module over
the \(\mathbb{F}_2\)-algebra generated by the Chern classes of \(A+B\).
Using the Whitney sum formula \(CH^*(BL)/2\) is then generated by elements of bounded degree as a module over
the \(\mathbb{F}_2\)-algebra generated by the first and the second Chern classes of \(A\) and \(B\).
Finally, applying \thref{thm:reg}, 
we now know that \(CH^*(BL)/2\) is generated by terms of degree at most 2 modulo relations of degree at most 4. These generators are Chern classes by \thref{thm:surjgrade}.
They can be assumed to be \[c_1(A), c_2(A),c_1(B),c_2(B),c_1(C),c_2(C)\] via \thref{Lem:deg1l} as all first Chern classes of degree 1 representations are generated
by the classes \(c_1(A), c_1(B), c_1(C)\) and \(R(L)\) is generated by \(A,B,C\) and the degree 1 representations.

The relations stated in this proposition are a generating set of the kernel 
\[\mathbb{F}_2[c_1(A),c_1(B),c_1(C),c_2(A),c_2(B),c_2(C)]\to CH^*(BC_2^{4,L})/2 \times \Pi_{i=0,1,\infty}CH^*(BH_i)/2\]
so there cannot be any more relations. Furthermore, by \thref{Prop:chldetect} all these relations must hold, as they
restrict to zero on the subgroups \(H_0,H_1,H_\infty, C_2^{4,L}\).
\end{proof}

\section{The Main Computation}\label{sec:chowg}
Denote \(G\coloneq \Syl_2\GL(4,2)\) which we identify with the group of upper triangular \(4\times 4\) \(\mathbb{F}_2\)-matrices
whose diagonal entries are all 1. As before we take finite groups to be finite discrete groups over \(\mathbb{C}\).
The goal of this section is to compute \(CH^*(BG)/2\) which is the main theorem of this text.

The computation is rather involved and it is easy to get lost among the many term manipulations involving \(\gamma\)-operations.
For this reason we give a complete roadmap of the journey ahead. We invite the reader to consult it whenever he feels lost.
Our computation will roughly abide by the following structure:
\begin{itemize}
\item We compute \(CH^i(BG)\cong \gr^i_\mathrm{geom} R(G)\) with \(i=1,2\) as in the previous subsection using the \(\gamma\)-filtration on \(R(G)\).
We show that in these degrees the associated graded to the \(\gamma\)- and the geometric filtration coincide and that (modulo 2) cycles are detected on \(L_0,L_{(x,z)}\)
as defined in the previous section.

    \item \thref{Prop:lowfilt} and the previous step tell us that \(\gr_\gamma^3R(G)\to \gr_{\mathrm{geom}}^3R(G)\)
    is surjective. In particular \(\gr_{\mathrm{geom}}^3R(G)\) is generated by Chern classes.
    Using universal polynomials, we can prove that \(\gr_{\mathrm{geom}}^3R(G)\) is 2-torsion.
   
    \item Next we want to determine \(CH^3(BG)\) by identifying all degree 3 relations between Chern classes.
    The map \(CH^3(BG) \to \gr_{\mathrm{geom}}^3R(G)\) is surjective,
    and the cardinality of \(H^3(BG,\mathbb{Z})\) gives an upper bound on the cardinality of the kernel
    as described in \thref{lem:surjgrade3}. Take \(A^3\) to be the free \(\mathbb{F}_2\)-vector
    space generated by degree 3 products of Chern classes.
    We identify a subgroup of said kernel whose size meets this upper bound by finding sufficiently many elements in the kernel of
    \[f\colon A^3\to \gr^3_\gamma R(G)\]
    that do not vanish under the map
    \[g\colon A^3\to CH^3(BG)\to\prod_{(x,z)\in\mathbb{F}_2^2-\{(0,0)\}}  CH^3(BL_{(x,z)}).\]

    We will also see that the map \(CH^3(BG) \to \gr_{\mathrm{geom}}^3R(G)\) factors through \(\gr^3_\gamma R(G)\) such that all elements in the kernel of \(CH^3(BG) \to \gr_{\mathrm{geom}}^3R(G)\) already vanish in \(\gr_{\gamma}^3R(G)\). This then tells us that \(\gr_{\mathrm{geom}}^3R(G)\cong \gr_{\gamma}^3R(G)\).
    Furthermore, we get that the relations holding in \(CH^3(BG)\)
are simply given by the intersection of the kernels of \(f\) and \(g\). Building on this, it is again easy to check that (mod 2) degree 3 cycles are detected on the \(L_{(x,z)}\).
    \item We apply \thref{Kor:detect3} to get that 
    \[CH^*(BG)/2\to\left(\prod_{(x,z)\in\mathbb{F}_2^2-\{(0,0)\}}  CH^*(BL_{(x,z)})/2\right)\times\left(\prod_{i\in \{0,\infty\}}  CH^*(BL_{i})/2\right)\]
    is injective.
    Thus we can determine all relations governing Chern classes by calculating the kernel of 
    \[\mathbb{F}_2[\text{Chern classes}]\to CH^*(BG)/2\to\left(\prod_{(x,z)\in\mathbb{F}_2^2-\{(0,0)\}}  CH^*(BL_{(x,z)})/2\right)\times\left(\prod_{i\in \{0,\infty\}}  CH^*(BL_{i})/2\right).\] 
    \item We use our work in \autoref{sec:chowh} and \autoref{sec:chowl} together with the fact that mod 2 cycles are detected on the \(L_i\) to show that the mod 2 cycle class map is injective for \(G\). By the computation of \(H^*(BG,\mathbb{Z}/2\mathbb{Z})\) due to Green and King in \cite{cohom} and the injectivity of the mod 2 cycle class map, we determine the ranks of \(CH^4(BG)/2\) and \(CH^5(BG)/2\). From this we conclude that all generators up to degree 5 can be assumed to be Chern classes. Finally, by \thref{thm:reg} we get that \(CH^*(BG)/2\) is generated by elements up to degree 5, meaning it is generated by Chern classes.

\end{itemize}

\subsection{Representations of \(\Syl_2\GL(4,2)\)}\label{subsec:repg}
As before, we base all our computations on the representation ring \(R(G)\). 
Denote in the sequel by \(E_{i,j}\in G\) the matrix whose diagonal entries and \((i,j)\)-th entry are \(1\) while all other entries are \(0\). Note that \(G\) is generated by \(E_{1,2},E_{2,3},E_{3,4}\).

Define for \((a,b,c)\in\mathbb{F}_2^3\)
\begin{align*}
    g_{(a,b,c)}\colon G&\to \mathbb{C}^\times\\
    \begin{pmatrix} 
        1&x&*&*\\
        0&1&y&*\\
        0&0&1&z\\
        0&0&0&1\\
    \end{pmatrix}&\mapsto (-1)^{ax+by+cz}.
\end{align*} 
There are surjective homomorphisms \(G\to D_8\coloneq \langle x,y\mid x^4=y^2=1,yxy^{-1}=x^{-1}\rangle\) given by
\begin{align*}
    \alpha_0\colon& E_{2,3}E_{1,2}\to x, E_{2,3}\to y, E_{3,4}\to 1\\
    \alpha_1\colon& E_{2,3}E_{1,2}\to x, E_{2,3}\to y, E_{3,4}E_{1,2}\to 1\\
    \alpha_\infty\colon& E_{1,2}\to 1,E_{3,4}E_{2,3}\to x, E_{3,4}\to y.
\end{align*}
Recall the standard representation \(V\) of \(D_8\) from \thref{Prop:repdih} and define for \(nk^{-1}\in\mathbb{F}_2\cup\{\infty\}\)
\begin{align*}
    \phi_{nk^{-1}}\colon G&\to\GL(2,\mathbb{C}), \phi_{nk^{-1}}=V\circ \alpha_{nk^{-1}}.
\end{align*}
Lastly define the faithful representation
\begin{gather*}
    \psi\colon G\to\GL(4,\mathbb{C})\\
    E_{1,2}\mapsto \begin{pmatrix}
        0&1&0&0\\1&0&0&0\\0&0&0&1\\0&0&1&0
    \end{pmatrix}, E_{2,3}\mapsto \begin{pmatrix}1 & 0&0&0\\0&1&0&0\\ 0& 0&1&0\\0&0&0&-1\end{pmatrix}, E_{3,4}\mapsto\begin{pmatrix}
        0&0&1&0\\0&0&0&1\\1&0&0&0\\0&1&0&0
    \end{pmatrix}.
\end{gather*}
\begin{Prop}\thlabel{prop:repg2}
    The representation ring \(R(G)\) can be expressed via a presentation
    \[R(G)\cong\mathbb{Z}[g_{(a,b,c)},\phi_{nk^{-1}},\psi\mid (a,b,c)\in\mathbb{F}_2^3,nk^{-1}\in\mathbb{F}_2\cup\{\infty\}]/I\]
    where \(I\) is generated by the relations

    \begin{equation}
        g_{(0,0,0)}=1,g_{(a,b,c)}\otimes g_{(a',b',c')}=g_{(a+a',b+b',c+c')},
    \end{equation}
    \begin{equation}\label{eq:gphi}
        g_{(ka,b,na)}\otimes \phi_{nk^{-1}}=\phi_{nk^{-1}} \text{ for } a,b\in\mathbb{F}_2,
    \end{equation}
    \begin{equation}\label{eq:phisq}
        \phi_{nk^{-1}}^2=\sum_{a,b\in\mathbb{F}_2}g_{(ka,b,na)},
    \end{equation}
    for \(\ell_1,\ell_2,\ell_3\in\{0,1,\infty\}\) pairwise distinct and \(a,b,c\in\mathbb{F}_2\)
    such that \(\phi_{\ell_3}\neq g_{(a,b,c)}\otimes\phi_{\ell_3}\),
    \[\phi_{\ell_1}\otimes\phi_{\ell_2}=\phi_{\ell_3}+g_{(a,b,c)}\otimes\phi_{\ell_3}\]
    and
    \begin{equation}\label{eq:acpsi}
        g_{(a,0,c)}\otimes \psi=\psi
    \end{equation}
    \begin{equation}\label{eq:psiphi}
        \phi_{nk^{-1}}\otimes\psi=\psi+g_{(0,1,0)}\otimes \psi,
    \end{equation}
    \begin{equation}
        \psi^2=3\psi+g_{(0,1,0)}\otimes\psi.
    \end{equation}
    
    Furthermore, the \(\lambda\)-structure on \(R(G)\) is determined by the \(\lambda\)-operations
    \begin{equation}\label{eq:lphip}
        \lambda^2(\phi_0)=g_{(1,1,0)}, \lambda^2(\phi_\infty)=g_{(0,1,1)},
    \end{equation}
    \begin{equation}\label{eq:lpsip}
         \lambda^2(\psi)=g_{(0,0,1)}\otimes\phi_0+g_{(0,0,1)}\otimes\phi_1+g_{(1,0,0)}\otimes\phi_\infty,
    \end{equation} 
    \begin{equation}\label{eq:l3} 
        \lambda^3(\psi)=g_{(0,1,0)}\otimes\psi,
    \end{equation}
    \begin{equation}\label{eq:l4}
         \lambda^4(\psi)=g_{(0,1,0)}.
    \end{equation}
 \end{Prop}
 \begin{proof}
    This can be deduced from the character table provided by \cite{GAP4}.
    To determine the \(\lambda\)-structure, one can use the \cite{GAP4}-function \texttt{AntiSymmetricParts}.
 \end{proof}
\subsection{Determining Degree 1}
\begin{Lem}\thlabel{lem:chgdeg1}
    \[CH^1(BG)=\langle c_1(\phi_0),c_1(\phi_\infty),c_1(\psi)\rangle_{\mathbb{Z}/2\mathbb{Z}}.\]
    Furthermore, (modulo 2) cycles of \(G\) of degree 1 are detected on the \(L_{(x,z)}\) with \((x,z)\in\mathbb{F}_2^2-\{(0,0)\}\).
\end{Lem}
\begin{proof}
    The presentation of \(CH^1(BG)\) follows from \thref{prop:deg1} together with the formulas for top exterior powers in \thref{prop:repg2}. The \enquote{furthermore} part is straightforward to check.
\end{proof}

\subsection{Determining Degree 2}\label{subsec:chgdeg2}
\begin{Lem}\thlabel{lem:chgdeg2}
    \begin{gather*}
    CH^2(BG)\cong \gr^2_{\mathrm{geom}} R(G)=\gr_\gamma^2R(G) \\\cong
    \frac{\langle c_1(\alpha)c_1(\beta)\mid \alpha,\beta\in\{\phi_0,\phi_\infty, \psi\}\rangle_{\mathbb{Z}/2\mathbb{Z}}}
    {\langle c_1(\phi_0)c_1(\psi)+c_1(\psi)^2, c_1(\phi_\infty)c_1(\psi)+c_1(\psi)^2\rangle_{\mathbb{Z}/2\mathbb{Z}}}\oplus\langle c_2(\phi_0),c_2(\phi_\infty), c_2(\psi)\rangle_{\mathbb{Z}/4\mathbb{Z}}.
    \end{gather*}
    Furthermore, 
    \[c_2(\phi_1)=c_1(\psi)^2+2c_2(\psi)-c_2(\phi_0)+c_2(\phi_\infty)+c_1(\phi_0)c_1(\phi_1)+c_1(\phi_0)c_1(\phi_\infty)+c_1(\phi_0)^2\]
    and the modulo 2 cycle class map \(CH^2(BG)/2\to H^4(BG,\mathbb{Z}/2\mathbb{Z})\) is injective.
\end{Lem}
\begin{proof}
Keep in mind in the following proof that due to \thref{thm:surjgrade} there is a natural
surjective ring map \(\gr^2_\gamma R(G)\to \gr^2_{\mathrm{geom}}R(G)\cong CH^2(BG)\).
This means that any relation in \(\gr^2_\gamma R(G)\) also holds in \(\gr^2_{\mathrm{geom}}R(G)\cong CH^2(BG)\)
but not necessarily the other way around.

We have seen in \thref{prop:repg2} that the representations \(g_{(a,b,c)},\phi_0,\phi_1,\phi_\infty,\psi\)
generate \(R(G)\) and thus their Chern classes generate \(\gr^*_\gamma R(G)\).
\thref{lem:chgdeg1} tells us that first Chern classes of the line elements \(g_{(a,b,c)}\) are already generated by the first Chern classes
of the \(\phi_i,\psi\). Therefore
\(\gr^2_{\mathrm{geom}}R(G)\cong CH^2(BG)\) is generated (as an abelian group) by twofold products of first Chern classes \(c_1(\phi_0),c_1(\phi_\infty),c_1(\psi)\) together with the second Chern classes of \(\phi_0,\phi_1,\phi_\infty\) and \(\psi\). 

We now show that \(c_2(\phi_0),c_2(\phi_\infty),c_2(\psi)\) are 4-torsion.
For this we will do some more involved computations and to make them more legible we denote
\(\widetilde{\phi_0}\coloneq\phi_0-2, \widetilde{\psi}\coloneq\psi-4\),
meaning that \(c_i(\phi_0)\) is the image of \(\gamma^i(\widetilde{\phi_0})\) in \(\gr_\gamma^i(R(G))\) and \(c_i(\psi)\)
is the image of \(\gamma^i(\widetilde{\psi})\). This is the first time we encounter computations in terms
of \(\gamma\)-operations so we will provide a lot of detail. Later we will omit details as the computation is straightforward, see \thref{rem}.
We can compute with the help of the Whitney sum formula
\begin{align*}
    4\gamma^2(\widetilde{\phi_0)}&=4\lambda^2(\phi_0-1)\\
    &=4(\lambda^2(\phi_0)+\lambda^1(\phi_0)\lambda^1(-1)+\lambda^2(-1))\\
    &=4(g_{(1,1,0)}-\phi_0+1) \text{ via \eqref{eq:lphip}}\\
    &=-g_{(0,1,0)}\psi+\psi+g_{(0,1,0)}\psi-\psi+4(g_{(1,1,0)}-\phi_0+1)\\
    &=-g_{(1,1,0)}\psi+\psi+g_{(0,1,0)}\psi-\psi+4(g_{(1,1,0)}-\phi_0+1) \text{ via \eqref{eq:acpsi}}\\
    &=-\psi(g_{(1,1,0)}+\phi_0-1)+4(g_{(1,1,0)}-\phi_0+1) \text{ via \eqref{eq:psiphi}}\\
    &=-(\psi-4)(g_{(1,1,0)}-\phi_0+1)\\
    &=-\gamma^1(\widetilde{\psi)}\gamma^2(\widetilde{\phi_0})\in \Gamma^3.\\
\end{align*}
Thus \(c_2(\phi_0)\) is 4-torsion. To show that \(c_2(\phi_\infty)\) is 4-torsion simply substitute \(\phi_\infty\) for \(\phi_0\) and \(g_{(0,1,1)}\) for \(g_{(1,1,0)}\) in the identities above.

The remaining relations follow via universal polynomials in the sense of \thref{prop:basicpoly}:
\begin{itemize}
\item The second Chern class of the exterior power \[\lambda^2(\psi)=g_{(0,1,1)}\otimes\phi_0+g_{(0,1,1)}\otimes\phi_1+g_{(1,1,0)}\otimes\phi_\infty\]
determines the relation for \(c_2(\phi_1)\).

 \item The second Chern classes of the tensor products \[\phi_0\otimes\psi=\psi+ g_{(n,1,k)}\otimes\psi\text{ with }n,k\in\mathbb{F}_2\]
 determine the two relations for \(c_1(\psi)^2\) as well as the 4-torsion of \(c_2(\psi)\).
 \end{itemize}

It only remains to show that we have found all relations and that modulo 2 cycles are detected on the \(L_{(x,z)}\) with \((x,z)\neq (0,0)\). Both can be shown by restricting to the \(CH^2(BL_{(x,z)})\) and \(CH^2(BL_{(x,z)})/2\) respectively.

\end{proof}

\subsection{Determining Degree 3}\label{subsec:syldeg3}
\begin{Lem}\thlabel{lem:chgdeg3}
    The abelian group \(CH^3(BG)\) is a free \(\mathbb{Z}/2\mathbb{Z}\)-module generated by a basis
    \begin{gather*}
    c_3(\psi),\\ c_2(\phi_0)c_1(\phi_0), c_2(\phi_0)c_1(\psi),\\
    c_2(\phi_\infty)c_1(\phi_0), c_2(\phi_\infty)c_1(\phi_\infty), c_2(\phi_\infty)c_1(\psi), \\ 
    c_2(\psi)c_1(\phi_0), c_2(\psi)c_1(\phi_\infty),
    c_2(\psi)c_1(\psi),\\
    c_1(\phi_0)^3, c_1(\phi_0)^2c_1(\phi_\infty), c_1(\phi_0)c_1(\phi_\infty)^2.
\end{gather*}
    Furthermore, we can express \(c_1(\phi_\infty)c_2(\phi_0),c_1(\phi_\infty)^3,c_1(\psi)^3\) via relations
    \begin{gather*}
        c_1(\phi_\infty)c_2(\phi_0)=c_1(\phi_0)c_2(\phi_\infty)+c_1(\phi_0)c_2(\psi)+c_1(\phi_\infty)c_2(\psi),\end{gather*}
    \begin{gather*}
        c_1(\phi_\infty)^3=c_1(\phi_0)c_2(\phi_0)+c_1(\psi)c_2(\phi_0)+c_1(\phi_\infty)c_2(\phi_\infty)+c_1(\psi)c_2(\phi_\infty)\\+c_1(\phi_0)c_2(\psi)+c_1(\phi_\infty)c_2(\psi)+c_1(\phi_0)^3,
    \end{gather*}
    \begin{gather*}
    c_1(\psi)^3=c_1(\phi_\infty)c_2(\phi_\infty)+c_1(\psi)c_2(\phi_\infty)+c_1(\phi_\infty)c_2(\psi)+c_1(\psi)c_2(\psi)\\+c_1(\phi_0)^2c_1(\phi_\infty)+c_1(\phi_0)c_1(\phi_\infty)^2+c_1(\phi_\infty)^3.
    \end{gather*}
    Recall the subgroups \(L_{(x,z)},L_0\leq G\) from \thref{Lem:centr}.
    The combined restriction map
    \[CH^3(BG)/2\to CH^3(BL_0)/2\times\prod_{i\in\mathbb{F}_2^2-\{(0,0)\}} CH^3(BL_i)/2\] 
    is injective and (modulo 2) cycles of \(G\) of degree 3 are detected on the \(L_0,L_{(x,z)}\) with \((x,z)\in\mathbb{F}_2^2-\{(0,0)\}\).
\end{Lem}
\begin{proof}
Denote by \(\mathrm{Chern}^*(BG)\leq CH^*(BG)\) the subring generated by Chern classes.
We first compute \(\mathrm{Chern}^3(BG)\leq CH^3(BG)\) and then show \(\mathrm{Chern}^3(BG)= CH^3(BG)\).
There is only one new generator (as a ring) of \(\mathrm{Chern}^*(BG)\) in degree 3 and that is \(c_3(\psi)\).
We will immediately show that it is 2-torsion.
Recall from \eqref{eq:l3} the exterior power
\[\lambda^3(\psi)=g_{(0,1,0)}\otimes\psi.\]
The corresponding universal polynomials (in the sense of \thref{prop:basicpoly}) for determining the third Chern class of a third exterior power on the left hand side
and for determining the third Chern class of a tensor product on the right hand side yield a relation
\[c_1(\psi)^3+2c_1(\psi)c_2(\psi)+3c_3(\psi)=c_3(\psi)+2c_2(\psi)c_1(g_{(0,1,0)})+3c_1(\psi)c_1(g_{(0,1,0)})^2.\]
This relation reduces modulo the degree 1 and 2 relations from \thref{lem:chgdeg1} and \thref{lem:chgdeg2} to
\[3c_3(\psi)=c_3(\psi).\] As all first Chern classes are also 2-torsion, \(\mathrm{Chern}^3(BG), gr\gr^3_\gamma R(G),\gr^3_{\mathrm{geom}}R(G)\) are all 2-torsion.

Denote by \(A^3\) the free \(\mathbb{F}_2\)-vector space generated by the basis
\begin{align*}
    (&c_3(\psi),\\ &c_2(\phi_0)c_1(\phi_0), \;c_2(\phi_0)c_1(\phi_\infty), \;c_2(\phi_0)c_1(\psi),\\
    &c_2(\phi_\infty)c_1(\phi_0), \;c_2(\phi_\infty)c_1(\phi_\infty), \;c_2(\phi_\infty)c_1(\psi), \\ 
    &c_2(\psi)c_1(\phi_0), \;c_2(\psi)c_1(\phi_\infty),
    \;c_2(\psi)c_1(\psi),\\
    &c_1(\phi_0)^3, \;c_1(\phi_0)^2c_1(\phi_\infty), \;c_1(\phi_0)c_1(\phi_\infty)^2, \;c_1(\phi_\infty)^3, \;c_1(\psi)^3).
\end{align*}
Note that we excluded the products of first Chern classes where at least one factor is \(c_1(\psi)\) except for
\(c_1(\psi)^3\). This is done because they all coincide in \(CH^3(BG)\)
by the degree 2 relations we already found in \thref{lem:chgdeg2}.

Using the restrictions of representations to \(L_0, L_{(x,z)}\) with \((x,z)\in\mathbb{F}_2^2-\{(0,0)\}\)
and \cite{SAGEMATH}
we can determine the kernel \(M\) of the map
\[A^3\to \mathrm{Chern}^3(BG)\cong \mathrm{Chern}^3(BG)/2\to CH^3(BL_0)/2\times \prod_{i\in\mathbb{F}_2^2-\{(0,0)\}}CH^3(BL_i)/2.\]
We see \(M\) is spanned by the three terms
    \begin{gather*}
        c_1(\phi_0)c_2(\phi_\infty)+c_1(\phi_0)c_2(\psi)+
        c_1(\phi_\infty)c_2(\phi_0)+c_1(\phi_\infty)c_2(\psi),\end{gather*}
    \begin{gather*}
        c_1(\phi_0)c_2(\phi_0)+c_1(\psi)c_2(\phi_0)+c_1(\phi_\infty)c_2(\phi_\infty)+c_1(\psi)c_2(\phi_\infty)\\+c_1(\phi_0)c_2(\psi)+c_1(\phi_\infty)c_2(\psi)+c_1(\phi_0)^3+c_1(\phi_\infty)^3,
    \end{gather*}
    \begin{gather*}
        c_1(\phi_\infty)c_2(\phi_\infty)+c_1(\psi)c_2(\phi_\infty)+c_1(\phi_\infty)c_2(\psi)+c_1(\psi)c_2(\psi)\\+c_1(\phi_0)^2c_1(\phi_\infty)+c_1(\phi_0)c_1(\phi_\infty)^2+c_1(\phi_\infty)^3+c_1(\psi)^3.
    \end{gather*}
    We know \(\mathrm{Chern}^3(BG)\cong \mathrm{Chern}^3(BG)/2\) is an extension of \(A^3/M\). It remains to show that the terms in \(M\) actually
vanish in \(CH^3(BG)\) and thus vanish in  \(\mathrm{Chern}^3(BG)\). For this we first compute a submodule \(N\leq A^3\) that vanishes under the map
\(A^3\to CH^3(BG)\to \gr^3_\mathrm{geom} R(G)\) such that \(M\leq N\).

We will now verify that the following terms vanish in \(\gr^3_\gamma\) (thus also in \(\gr^3_\mathrm{geom}\)) and they will be the generators of \(N\):
\begin{equation}\label{eq:new1}
    c_1(\phi_0)c_2(\phi_0), 
\end{equation}
\begin{equation}\label{eq:new2}
    c_1(\phi_\infty)c_2(\phi_\infty),
\end{equation}
\begin{equation}\label{eq:new3}
    c_1(\psi)c_2(\psi)+c_3(\psi),
\end{equation}
\begin{equation}\label{eq:new4}
    c_1(\phi_0)^2c_1(\phi_\infty)+c_1(\phi_0)c_1(\phi_\infty)^2,
\end{equation}
\begin{equation}\label{eq:new5}
    c_1(\phi_\infty)c_2(\phi_0)+c_1(\psi)c_2(\phi_0)+
    c_1(\phi_0)c_2(\phi_\infty)+c_1(\psi)c_2(\phi_\infty),
\end{equation}
\begin{equation}\label{eq:new6}
    c_1(\phi_0)^3+c_1(\psi)^3+c_2(\phi_0)c_1(\psi)+c_2(\psi)c_1(\phi_0)+c_3(\psi),
\end{equation}
    \begin{equation}\label{eq:new7}
     c_1(\phi_\infty)^3+c_1(\psi)^3+c_2(\phi_\infty)c_1(\psi)+c_2(\psi)c_1(\phi_\infty)+c_3(\psi).
\end{equation}

For better readability we denote
\[\phi'_{nk^{-1}}\coloneq g_{(k+1,0,n)}\otimes \phi_{nk^{-1}},\psi'\coloneq g_{(0,1,0)}\otimes\psi\]
such that the group action by tensor multiplication with the \(g_{(a,b,c)}\) has (among others) orbits \(\{\phi_{nk^{-1}},\phi'_{nk^{-1}}\}, \{\psi,\psi'\}\).
For any representation \(X\) denote \[\widetilde{X}\coloneq X-\deg{X}\] so that the Chern class \(c_n(X)\in \gr^*_\gamma R(G)\) is the image of \(\gamma^n(\widetilde{X})\).

Following \thref{rem} we get
\begin{align*}
\gamma^1(\widetilde{\phi_0})\gamma^2(\widetilde{\phi_0})=&(\phi_0-2)(g_{(1,1,0)}-\phi_0+1)\text{ via the Whitney sum formula and \eqref{eq:lphip}}\\
=&\phi_0-(1+g_{(1,0,0)}+g_{(0,1,0)}+g_{(1,1,0)})+\phi_0-2(g_{(1,1,0)}-\phi_0+1)\text{ via \eqref{eq:phisq} and \eqref{eq:gphi}}\\
=&-(1-\phi_0+g_{(1,1,0)})+(\phi_0-(1+g_{(1,0,0)}+g_{(0,1,0)}+g_{(1,1,0)})+\phi_0)\\&-(g_{(1,1,0)}-\phi_0+1)\\
=&-(g_{(1,1,0)}-\phi_0+1)^2 \text{ via \eqref{eq:gphi} and \eqref{eq:phisq}}\\
=&-\gamma^2(\widetilde{\phi_0})^2\in\Gamma^4 \text{ via the Whitney sum formula and \eqref{eq:lphip}.} 
\end{align*}
Thus \(c_1(\phi_i)c_2(\phi_i)=0\) in \(\gr^3_\gamma R(G)\) for \(i=0\).
The cases \(i=1,\infty\) are analogous.
We get the relations \eqref{eq:new1}, \eqref{eq:new2} and also the relation
\begin{equation}\label{eq:new1mod}
    c_1(\phi_1)c_2(\phi_1)=0
\end{equation}
that we will need to prove \eqref{eq:new5}.

Verifying the relation \eqref{eq:new3} is more involved and we will take some intermediate steps.
Unfolding on both sides with the Whitney sum formula one can check
\begin{align*}
    \gamma^1(\widetilde{\psi})\gamma^2(\widetilde{\psi})+\gamma^3(\widetilde{\psi})=&-\gamma^2(\widetilde{\psi})^2+10\gamma^4(\widetilde{\psi)}-(\gamma^3(\widetilde{\psi})-\gamma^3(\widetilde{\psi'}))-\gamma^1(\widetilde{g_{(0,1,0)}})^3\\&+\gamma^4(\widetilde{\psi)}(\gamma^1(\widetilde{g_{(1,0,0)}})+\gamma^1(\widetilde{g_{(1,0,1)}})+\gamma^1(\widetilde{g_{(0,0,1)}}))\\
    \in&(-(\gamma^3(\widetilde{\psi})-\gamma^3(\widetilde{\psi'}))-\gamma^1(\widetilde{g_{(0,1,0)}})^3)+\Gamma^4.
\end{align*}
This computation step is quite cumbersome so the author gives a more detailed account in the proof of \cite[Lemma 3.16]{zphd}.
By this observation,  the relation \eqref{eq:new3} is equivalent (in \(\gr^*_\gamma R(G)\)) to the relation 
\(c_3(\psi)+c_1(g_{(0,1,0)})^3=c_3(\psi').\) This equivalent relation follows directly from the
universal polynomial for the third Chern class of the tensor product \(g_{(0,1,0)}\otimes\psi=\psi'\): 
\[c_3(\psi')=c_3(g_{(0,1,0)}\otimes\psi)=4c_1(g_{(0,1,0)})^3+3c_1(g_{(0,1,0)})^2c_1(\psi)+2c_1(g_{(0,1,0)})c_2(\psi)+c_3(\psi)\]
together with the relation \(c_1(\psi)=c_1(g_{(0,1,0)})\) from \thref{lem:chgdeg1} and the fact that \(\mathrm{Chern}^3(BG)\) is 2-torsion. This confirms the relation \eqref{eq:new3}.

To prove the relation \eqref{eq:new4} first check
\begin{gather*}
    \gamma^1(\widetilde{\phi_0})=\gamma^1(\widetilde{g_{(1,1,0)}})-\gamma^2(\widetilde{\phi_0}),\\
    \gamma^1(\widetilde{\phi_\infty})=\gamma^1(\widetilde{g_{(0,1,1)}})-\gamma^2(\widetilde{\phi_\infty}),\\
    -2\gamma^1(\widetilde{g_{(1,1,0)}})=\gamma^1(\widetilde{g_{(1,1,0)}})^2,\\
    -2\gamma^1(\widetilde{g_{(0,1,1)}})=\gamma^1(\widetilde{g_{(0,1,1)}})^2,
\end{gather*}
and from this deduce modulo \(\Gamma^4\)
\begin{align*}
    \gamma^1(\widetilde{\phi_0})\gamma^1(\widetilde{\phi_\infty})^2+\gamma^1(\widetilde{\phi_0})^2\gamma^1(\widetilde{\phi_\infty})
    &\equiv \gamma^1(\widetilde{g_{(1,1,0)}})\gamma^1(\widetilde{g_{(0,1,1)}})^2
    +\gamma^1(\widetilde{g_{(1,1,0)}})^2\gamma^1(\widetilde{g_{(0,1,1)}})\\
    &\equiv -4\gamma^1(\widetilde{g_{(1,1,0)}})\gamma^1(\widetilde{g_{(0,1,1)}})\\
    &\equiv -\gamma^1(\widetilde{g_{(1,1,0)}})^3\gamma^1(\widetilde{g_{(0,1,1)}})\equiv 0\mod{\Gamma^4}.
\end{align*}
Thus we have shown the relation \eqref{eq:new4}.

The relation \eqref{eq:new5} will be verified using \eqref{eq:new1mod}
\[c_1(\phi_1)c_2(\phi_1)=0.\]

We can expand the left hand side of the equation using the relations from \thref{lem:chgdeg1} and \thref{lem:chgdeg2}
\begin{align*}
c_2(\phi_1)&=c_1(\psi)^2+2c_2(\psi)-c_2(\phi_0)+c_2(\phi_\infty)+c_1(\phi_0)c_1(\phi_1)+c_1(\phi_0)c_1(\phi_\infty)+c_1(\phi_0)^2\\
        c_1(\phi_1)&=c_1(\phi_0)+c_1(\phi_\infty)+c_1(\psi)
\end{align*}
to get a relation without \(c_2(\phi_1)\) and \(c_1(\phi_1)\)
\[(c_2(\phi_0)+c_2(\phi_\infty)+c_1(\phi_0)(c_1(\phi_0)+c_1(\phi_\infty)+c_1(\psi))+c_1(\psi)^2)(c_1(\phi_0)+c_1(\phi_\infty)+c_1(\psi))=0.\]
Using \eqref{eq:new4} and the following relations from \thref{lem:chgdeg2}
\[c_1(\phi_0)c_1(\psi)=c_1(\phi_\infty)c_1(\psi)=c_1(\psi)^2\]
this reduces to
\[(c_2(\phi_0)+c_2(\phi_\infty))(c_1(\phi_0)+c_1(\phi_\infty)+c_1(\psi))=0.\]
Finally we apply \eqref{eq:new1}, \eqref{eq:new2} to reduce this to the relation \eqref{eq:new5}.

We see that relation \eqref{eq:new6} is equivalent to the relation
\[(c_1(\phi_0)+c_1(\psi))(c_1(\phi_0)^2+c_2(\phi_0)+c_2(\psi))\]
by \eqref{eq:new1}, \eqref{eq:new3} and the relation \(c_1(\phi_0)c_1(\psi)=c_1(\psi)^2\) from \thref{lem:chgdeg2}.

To prove this equivalent relation, we first use the Whitney sum formula and relations from \thref{prop:repg2} to check
\begin{align*}
    \gamma^2(\widetilde{\psi})&=\gamma^4(\widetilde{\psi})-2\gamma^1(\widetilde{\psi})+\gamma^1(\widetilde{\psi'})-\gamma^1(\widetilde{g_{(0,1,0)}})\\
    &\equiv -2\gamma^1(\widetilde{\psi})+\gamma^1(\widetilde{\psi'})-\gamma^1(\widetilde{g_{(0,1,0)}})\mod{\Gamma^4}.
    \end{align*}
In \thref{lem:chgdeg1} we have seen \(c_1(\phi_0)=c_1(g_{(1,1,0)})\) and \(c_1(\psi)=c_1(g_{(0,1,0)})=-c_1(g_{(0,1,0)})\)
which means
\[\gamma^1(\widetilde{\phi_0})+\gamma^1(\widetilde{\psi})\equiv\gamma^1(\widetilde{g_{(1,1,0)}})-\gamma^1(\widetilde{g_{(0,1,0)}})\mod{\Gamma^2}.\]
From these two observations, we can conclude
\begin{gather*}
    (\gamma^1(\widetilde{\phi_0})+\gamma^1(\widetilde{\psi}))(\gamma^1(\widetilde{\phi_0})^2+\gamma^2(\widetilde{\psi})+\gamma^2(\widetilde{\phi_0}))\equiv\\
   (\gamma^1(\widetilde{g_{(1,1,0)}})-\gamma^1(\widetilde{g_{(0,1,0)}})
    (\gamma^1(\widetilde{\phi_0})^2+\gamma^2(\widetilde{\phi_0})-2\gamma^1(\widetilde{\psi})+\gamma^1(\widetilde{\psi'})-\gamma^1(\widetilde{g_{(0,1,0)}}))\mod{\Gamma^4}.
\end{gather*}
Now we can use the Whitney sum formula and relations from \thref{prop:repg2} to see
\begin{align*}
    (\gamma^1(\widetilde{g_{(1,1,0)}})+\gamma^1(\widetilde{g_{(0,1,0)}})
    (\gamma^1(\widetilde{\phi_0})^2+&\gamma^2(\widetilde{\phi_0})-2\gamma^1(\widetilde{\psi})+\gamma^1(\widetilde{\psi'})-\gamma^1(\widetilde{g_{(0,1,0)}}))\\&=10(g_{(1,1,0)}-g_{(0,1,0)})+2(1-g_{(1,0,0)})\\
    &=8(g_{(1,1,0)}-1)-8(g_{(0,1,0)}-1)+2(g_{(1,0,0)}-1)(g_{(0,1,0)}-1)\\
    &=\gamma^1(\widetilde{g_{(1,1,0)}})^4-\gamma^1(\widetilde{g_{(0,1,0)}})^4-\gamma^1(\widetilde{g_{(1,0,0)}})\gamma^1(\widetilde{g_{(0,1,0)}})^2\\
    &\equiv \gamma^1(\widetilde{g_{(1,0,0)}})\gamma^1(\widetilde{g_{(0,1,0)}})^2\mod \Gamma^4.
\end{align*}
Using \thref{lem:chgdeg1}, we can conclude from this that \eqref{eq:new6} is equivalent to the relation
\[c_1(\psi)^2(c_1(\psi)+c_1(\phi_0)),\] which vanishes by \thref{lem:chgdeg2}.
 
The final relation \eqref{eq:new7} holds analogously to \eqref{eq:new6}.

We have that \(M\leq N\) and because \(N\) is a vector space there is a splitting \(N=M\oplus O\).
As \(N\) is of dimension 7 and \(M\) of dimension \(3\) we know \(O\) is of dimension \(4\).
Furthermore, as \(\mathrm{Chern}^3(BG)\cong \mathrm{Chern}^3(BG)/2\) is an extension of \(A^3/M\), the map \(O\to A^3\to \mathrm{Chern}^3(BG)\)
is injective. Using \cite{HAP} we can determine via the function \texttt{GroupCohomology} that
\(H^3(BG,\mathbb{Z})\cong (\mathbb{Z}/2\mathbb{Z})^4\) and so by \thref{lem:surjgrade3}
the kernel of \(CH^3(BG)\to \gr^3_\mathrm{geom}R(G)\) is a \(\mathbb{Z}/2\mathbb{Z}\)-vector space of
dimension at most 4. As the kernel of \(CH^3(BG)\to \gr^3_\mathrm{geom}R(G)\) contains the injective image 
of \(O\) as a submodule, it has to be equal to \(O\). Thus \(M\) has to vanish in \(CH^3(BG)\)
and we get \(\mathrm{Chern}^3(BG)\cong A^3/M\). By our definition of \(M\), mod 2 cycles in \(\mathrm{Chern}^3(BG)\cong A^3/M\)
are detected on \(L_0,L_{(x,z)}\) with \((x,z)\in\mathbb{F}_2^2-\{(0,0)\}\).

Furthermore, we have seen in \thref{lem:chgdeg2}
that \(\Gamma^3=F^3_\mathrm{geom}\) so \(\gr^3_\mathrm{geom}R(G)\)
is generated by products of Chern classes and thus \(\mathrm{Chern}^3(BG)\to \gr^3_\mathrm{geom}R(G)\) is surjective.
We know that the kernel of the surjective map \(CH^3(BG)\to \gr^3_\mathrm{geom}R(G)\) is \(O\) 
which is the same as the kernel of the surjective map \(\mathrm{Chern}^3(BG)\to \gr^3_\mathrm{geom}R(G)\)
so \(CH^3(BG)=\mathrm{Chern}^3(BG)\). 
As algebraic cycles in \(\mathrm{Chern}^3(BG)/2\) are detected on the \(L_0,L_{(x,z)}\) with \((x,z)\in\mathbb{F}_2^2-\{(0,0)\}\) so are algebraic cycles in \(CH^3(BG)/2\).
\end{proof}

\subsection{The Subring of Chern Classes in \(CH^*(B(\Syl_2\GL(4,2)))/2\)}\label{subsec:chernsyl}
We now compute the subring generated by Chern classes \(\mathrm{Chern}^*(BG)/2\leq CH^*(BG)/2\) using detection on subgroups.
In the next subsection we will see that \(\mathrm{Chern}^*(BG)/2=CH^*(BG)/2\). 

\begin{Thm}\thlabel{thm:resg}
    The combined restriction map
    \[CH^*(BG)/2\to \prod_i CH^*(BL_i)/2\]
    with \(i=0,\infty,(0,1),(1,1),(1,0)\) and \(L_i\) defined as in \thref{Lem:centr}
    is injective.
\end{Thm}
\begin{proof}
    We know from \thref{lem:chgdeg1}, \thref{lem:chgdeg2} and \thref{lem:chgdeg3} that cycles in
\(CH^i(BG)/2\) with \(i=1,2,3\) are detected on the family of subgroups \(L_0,L_\infty,L_{(x,z)}\) with \((x,z)\in\mathbb{F}_2^2-\{(0,0)\}\) defined in \thref{Lem:centr}.
Observe that \(G\) has a faithful representation \(\psi\) of degree 4
and the center \(Z(G)\) of \(G\) is of 2-rank 1.
Thus \thref{Kor:detect3} together with \thref{Lem:centr} tells us that all cycles in \(CH^*(BG)/2\) 
are detected on \(L_0,L_\infty,L_{(x,z)}\). 
\end{proof}
\begin{Kor}\thlabel{thm:injccm}
    The cycle class map
    \[CH^*(BG)/2\to H^{2*}(BG,\mathbb{Z}/2\mathbb{Z})\]
    is injective.
\end{Kor}
\begin{proof}
    This follows from \thref{thm:resg} analogously to the proof of \thref{cor:lcyc}.
\end{proof}
\begin{Thm}\thlabel{thm:cherng}
    The graded \(\mathbb{F}_2\)-algebra \(\mathrm{Chern}^*(BG)/2\) is generated by Chern classes
    \[c_i(A)\coloneq c_i(\phi_0),c_i(B)\coloneq c_i(\phi_\infty),c_j(V)\coloneq c_j(\psi)\] with \(i=1,2\) and \(j=1,\dots,4\)
    subject to the relations:
    \begin{itemize}
    \item in degree 2
    \[c_1(B)c_1(V)+c_1(V)^2, c_1(A)c_1(V)+c_1(V)^2,\]
    \item in degree 3
    \[c_1(B)c_2(A)+c_1(V)c_2(A)+c_1(A)c_2(B)+c_1(V)c_2(B),\]
    \begin{gather*}
        c_1(A)^2c_1(B)+c_1(A)c_1(B)^2+c_1(B)^3+c_1(V)^3\\+c_1(B)c_2(B)+c_1(V)c_2(B)+c_1(B)c_2(V)+c_1(V)c_2(V),
    \end{gather*}
    \begin{gather*}
        c_1(A)^3+c_1(B)^3+c_1(A)c_2(A)+c_1(V)c_2(A)+c_1(B)c_2(B)\\+c_1(V)c_2(B)+c_1(A)c_2(V)+c_1(B)c_2(V), 
    \end{gather*}
    \item in degree 4

    \begin{gather*}
        c_1(A)^2c_2(A)+c_1(V)^2c_2(A)+c_1(A)c_1(B)c_2(B) \\
        +c_1(B)^2c_2(B)+c_1(A)^2c_2(V)+c_1(A)c_1(B)c_2(V) \\
        +c_1(B)^2c_2(V)+c_1(V)^2c_2(V)\\
        +c_2(A)^2+c_2(A)c_2(B)+ c_2(B)^2+c_2(V)^2+c_1(V)c_3(V),
    \end{gather*}
    \begin{gather*}
        c_1(B)^4+c_1(V)^4+c_1(B)^2c_2(V)+c_1(V)^2c_2(V)+c_1(B)c_3(V)+c_1(V)c_3(V),
    \end{gather*}
    \begin{gather*}
    c_1(A)c_1(B)^3+c_1(V)^4+c_1(B)^2c_2(B)+c_1(V)^2c_2(B)+ c_1(A)^2c_2(V)\\
    +c_1(B)^2c_2(V)+c_2(A)^2+c_2(A)c_2(B)+ c_2(B)^2+c_2(V)^2+c_1(A)c_3(V) 
    \end{gather*}

\item in degree 6
\begin{gather*}
    c_1(B)^3c_3(V) + c_1(V)^3c_3(V) + c_2(A)^2c_2(B) + c_2(A)c_2(B)^2\\ 
    + c_1(A)c_2(A)c_3(V)+c_1(V)c_2(A)c_3(V) + c_1(A)c_2(B)c_3(V)+c_1(V)c_2(B)c_3(V)\\+c_1(B)c_2(V)c_3(V)+c_1(V)^2c_4(V)+c_3(V)^2. 
\end{gather*}
    \end{itemize}
\end{Thm}
\begin{proof}
    By \thref{thm:resg}, for a generating set \(S\) of \(\mathrm{Chern}^*(BG)/2\),
for example the Chern classes of a generating set of \(R(G)\), we can obtain a presentation of
\(\mathrm{Chern}^*(BG)/2\) by simply taking as relations the kernel of the restriction map
\[\mathbb{F}_2[S]\to\mathrm{Chern}^*(BG)/2\to \prod_i CH^*(BL_i)/2\]
with \(i=0,\infty,(0,1),(1,1),(1,0)\). By the relations determined in \thref{lem:chgdeg1} and \thref{lem:chgdeg2} we can take \(S\coloneq\{c_i(\phi_0),c_i(\phi_\infty),c_j(\psi)\mid i=1,2\text{ and }j=1,\dots,4\}\). Computing the kernel of the map above
using \cite{SAGEMATH}, we get the presentation stated in the theorem.
\end{proof}
\subsection{Determining a Generating Set}\label{subsec:genssyl}

The last remaining step is to show that \(CH^*(BG)/2\) is generated by Chern classes.
\begin{Prop}\thlabel{Prop:deg456}
    \[CH^4(BG)/2=\mathrm{Chern}^4(BG)/2,\text{ and }CH^5(BG)/2=\mathrm{Chern}^5(BG)/2.\]
\end{Prop} 
\begin{proof}
Let \(C_2^n\leq G\) be an elementary abelian subgroup. Then the subgroup of \(H^{2i}(BC_2^n,\mathbb{F}_2)\)
 generated by products of Chern classes is exactly the subgroup generated by squares. Thus the elements of the image of the map
\[CH^i(BG)/2\to H^{2i}(BG,\mathbb{F}_2)\to H^{2i}(BC_2^n,\mathbb{F}_2),\]
i.e. the elements of the image of the composite of the mod 2 cycle class map and the restriction to \(C_2^n\),
are all sums of squares. 
We know from \thref{thm:injccm}
that the mod 2 cycle class map
\[CH^{i}(BG)/2\to H^{2i}(BG,\mathbb{F}_2)\]
is injective. Furthermore, we get
from \thref{thm:cherng} that \(\mathrm{Chern}^4(BG)/2 \) is a \(\mathbb{Z}/2\mathbb{Z}\)-module
of rank 19 and that \(\mathrm{Chern}^5(BG)/2 \) is of rank 28. 
One can compute the mentioned ranks via \cite{SAGEMATH} as explained in \cite[Appendix D]{zphd}.

Thus it suffices to show that the image of the modulo 2 cycle class map in \(H^8(BG,\mathbb{F}_2)\)
is of rank at most 19 and the image in \(H^{10}(BG,\mathbb{F}_2)\) is of rank at most 28. By the argument above, it would even suffice to show that the subgroup of \(H^8(BG,\mathbb{F}_2)\) spanned by terms that restrict to sums of squares on all maximal elementary abelian subgroups is of rank at most 19 --- likewise of rank at most 28 for \(H^{10}(BG,\mathbb{F}_2)\). One can easily check this using \cite{SAGEMATH} and the computation of the group cohomology ring in \cite[Cohomology of group number 138 of order 64]{cohom} as explained in \cite[Appendix D]{zphd}.
\end{proof}
In preparation to apply \thref{thm:reg} we state an elementary lemma.
\begin{Lem}\thlabel{lem:gens}
Let \(R\) be a commutative ring, \(S\) a commutative \(R\)-algebra and \(T\leq S\) a subalgebra.
If \(T\) is generated as a module over \(R\) by a set \(t\subseteq T\)
and \(S\) is generated as a module over \(T\) by a set \(s\subseteq S\) then
\(S\) is generated as a module over \(R\) by \(s\cup t \cup t\cdot s\).
\end{Lem}
\begin{proof}
    For any term \(a\in S\) there exist \(s_1,\dots, s_n \in s\) and \(\lambda_1,\dots,\lambda_n\in T\)
    with \(a=\lambda_1s_1+\dots+\lambda_ns_n\). Furthermore, there exist \(t_{1,1},\dots,t_{n,m}\in t\)
    and \(\mu_{1,1},\dots,\mu_{n,m}\in R\) such that \(\lambda_i=\mu_{i,1}t_{i,1}+\dots+\mu_{i,m}t_{i,m},\)
    meaning \(a=\mu_{1,1}t_{1,1}s_1+\dots+\mu_{n,m}t_{n,m}s_n.\)
\end{proof}
\begin{Thm}\thlabel{thm:chggens}
    \[CH^*(BG)/2=\mathrm{Chern}^*(BG)/2\]
\end{Thm}
\begin{proof}
\thref{lem:chgdeg1}, \thref{lem:chgdeg2}, \thref{lem:chgdeg3}  and \thref{Prop:deg456} already determined all the generators up to degree 5. By
\thref{thm:reg} all generators have degree at most 5 if
\(CH^*(BG)/2\) is generated as a module over \[R\coloneqq\mathbb{F}_2[c_1(\phi_0),c_1(\phi_\infty),c_1(\psi),c_2(\phi_0),c_2(\psi),c_4(\psi)]\] by a set of elements whose degrees all share some upper bound. By \thref{thm:faithgens} and \thref{lem:gens} it suffices to show that
\(\mathrm{Chern}^*(BG)/2\) is generated as a module over \(R\) by the elements
\[1,c_3(\psi), c_3(\psi)c_2(\phi_\infty), c_2(\phi_\infty).\]
This is the case by the presentation determined in \thref{thm:cherng}, as the degree 4 and 6 relations
 let us substitute the terms  \(c_2(\phi_\infty)^ic_3(\psi)^j\) with \(i\geq 2\) or \(j\geq 2\) for multiples of positive degree elements of \(\mathbb{F}_2[c_1(\phi_0),c_1(\phi_\infty),c_1(\psi),c_2(\phi_0),c_2(\psi),c_4(\psi)]\). By repeated substitutions we can identify any term in \(\mathrm{Chern}^*(BG)/2\) with a term of the form
\(\lambda_1+\lambda_2c_2(\phi_\infty)+\lambda_3c_3(\psi)+\lambda_4c_2(\phi_\infty)c_3(\psi)\)
where we have \[\lambda_1,\lambda_2,\lambda_3,\lambda_4\in\mathbb{F}_2[c_1(\phi_0),c_1(\phi_\infty),c_1(\psi),c_2(\phi_0),c_2(\psi),c_4(\psi)].\]
This concludes the proof.
\end{proof}
   \sloppy
\printbibliography
\end{document}